\documentclass[11pt]{amsart}
\usepackage[centering,includeheadfoot,margin=2.5 cm]{geometry}
\usepackage{graphics,enumitem,epsfig,textcomp}
\usepackage{amsfonts,amsmath,amssymb,euscript,color,mathrsfs}

\usepackage{hyperref}
\numberwithin{equation}{section}
\newtheorem{theorem}{Theorem}%[section]

\newtheorem{lemma}{Lemma}%[section]
\newtheorem{proposition}{Proposition}%[section]
\newtheorem{remark}{Remark}%[section]

\def\d{\,\mathrm{d}}
\def\eps{\varepsilon}
\def\N{\mathbb{N}}
\def\R{\mathbb{R}}
\def\C{\hbox{\rlap{\kern.24em\raise.1ex\hbox
      {\vrule height1.3ex width.9pt}}C}}
\def\P{\hbox{\rlap{I}\kern.16em P}}
\def\Q{\hbox{\rlap{\kern.24em\raise.1ex\hbox
      {\vrule height1.3ex width.9pt}}Q}}
\def\M{\hbox{\rlap{I}\kern.16em\rlap{I}M}}
\def\Z{\hbox{\rlap{Z}\kern.20em Z}}
\def\({\begin{eqnarray}}
\def\){\end{eqnarray}}
\def\[{\begin{eqnarray*}}
\def\]{\end{eqnarray*}}
\def\part#1#2{\frac{\partial #1}{\partial #2}}

\def\grad{\nabla}
\def\Norm#1{\left\| #1 \right\|}
\def\pmb#1{\setbox0=\hbox{$#1$}
  \kern-.025em\copy0\kern-\wd0
  \kern-.05em\copy0\kern-\wd0
  \kern-.025em\raise.0433em\box0 }
\def\bar{\overline}

\def\tot#1#2{\frac{\d #1}{\d #2}} 
\def\laplace{\Delta}
\def\d{\,\mathrm{d}}
\def\N{\mathbb{N}}
\def\R{\mathbb{R}}

\def\A{\mathbb{A}}

\def\eps{\varepsilon}

\def\E{\mathcal{E}}

\def\P{\mathbb{P}}
\def\Q{\mathbb{Q}}
\def\X{\mathbb{X}}

\def\embdd{\hookrightarrow}
\def\wto{\rightharpoonup}

\begin{document}
%%%%%%%%%%%%%%%%
%\allowdisplaybreaks

%    +-------------+
%    +    TITLE    +
%    ---------------

\centerline{{\Large Rigorous Continuum Limit for the Discrete Network Formation Problem}}
%\centerline{{\huge for Biological Network Formation}}
\vskip 7mm

%    -----------------
%    +    AUTHORS    +
%    -----------------
\centerline{
{\large Jan Haskovec}\footnote{Mathematical and Computer Sciences and Engineering Division,
King Abdullah University of Science and Technology,
Thuwal 23955-6900, Kingdom of Saudi Arabia; 
{\it jan.haskovec@kaust.edu.sa}}\qquad
{\large Lisa Maria Kreusser}\footnote{Department of Applied Mathematics and Theoretical Physics (DAMTP), University of Cambridge, Wilberforce Road, Cambridge CB3 0WA, UK;
{\it L.M.Kreusser@damtp.cam.ac.uk}}\qquad
{\large Peter Markowich}\footnote{Mathematical and Computer Sciences and Engineering Division,
King Abdullah University of Science and Technology,
Thuwal 23955-6900, Kingdom of Saudi Arabia;
Faculty of Mathematics, University of Vienna, Oskar-Morgenstern-Platz 1, 1090 Vienna, Austria;
{\it peter.markowich@kaust.edu.sa; peter.markowich@univie.ac.at}}
}
\vskip 10mm

%    ------------------
%    +    ABSTRACT    +
%    ------------------

\noindent{\bf Abstract.}
Motivated by recent physics papers describing the formation of biological transport networks
we study a discrete model proposed by Hu and Cai consisting of an energy consumption function
constrained by a linear system on a graph.
For the spatially two-dimensional rectangular setting
we prove the rigorous continuum limit of the constrained energy functional
as the number of nodes of the underlying graph tends to infinity and the edge lengths shrink to zero uniformly.
The proof is based on reformulating the discrete energy functional as a sequence of integral functionals
and proving their $\Gamma$-converge towards a continuum energy functional.
\vskip 7mm

%    -------------------
%    +    KEY WORDS    +
%    -------------------
\vskip 7mm

\noindent{\bf Math. Class. No.:} 35K55; 92C42; 65M60
\vskip 7mm

%\tableofcontents

%%%%%%%%%%%%%%%%%%%%%%%%%%%%%%%%%
%% Introduction
%%%%%%%%%%%%%%%%%%%%%%%%%%%%%%%%%
\section{Introduction}\label{sec:Intro}
In this paper we derive the rigorous continuum limit of the
discrete network formation model of Hu and Cai \cite{Hu-Cai}.
The model is posed on an a priori given graph $\mathbb{G} = (\mathbb{V},\mathbb{E})$, consisting of the set of vertices (nodes) $\mathbb{V}$
and the set of unoriented edges (vessels) $\mathbb{E}$.
Any pair of vertices $i$, $j\in \mathbb{V}$ is connected by at most one edge $(i,j)\in \mathbb{E}$,
such that the corresponding graph $(\mathbb{V}, \mathbb{E})$ is connected. The lengths $L_{ij}>0$ of the vessels $(i,j)\in \mathbb{E}$
are given a priori and fixed.
The adjacency matrix of the graph $(\mathbb{V},\mathbb{E})$ is denoted by $\A$, i.e., $\A_{ij}=1$ if $(i,j)\in \mathbb{E}$, otherwise $\A_{ij}=0$.

Let us emphasize that by fixing $(\mathbb{V},\mathbb{E})$, the set of possible flow directions in the network is also fixed.
For each node $j\in \mathbb{V}$ we prescribe the strength of source/sink $S_j\in\R$ and we adopt the convention
that $S_j>0$ denotes sources, while $S_j<0$ sinks. We also allow for $S_j=0$, i.e., no external
in- or outgoing flux in this node. We impose the global mass conservation
\(   \label{assS}
   \sum_{j\in \mathbb{V}} S_j = 0.
\)
We denote $C_{ij}$ and, resp., $Q_{ij}$ the conductivity and, resp., the flow through the vessel $(i,j)\in \mathbb{E}$.
Note that the flow is oriented and we adopt the convention that $Q_{ij}>0$ means
net flow from node $j\in \mathbb{V}$ to node $i\in \mathbb{V}$.
An overview of the notation is provided in Table \ref{table:notation}.
We assume low Reynolds number of the flow through the network, so that
the flow rate through a vessel $(i,j)\in \mathbb{E}$ is proportional to its conductivity and the pressure drop between its two ends, i.e.,
\(  \label{Q_i}
   Q_{ij} = C_{ij} \frac{P_j - P_i}{L_{ij}}.
\)
Local conservation of mass is expressed in terms of the Kirchhoff law,
\(  \label{Kirchhoff0}
   \sum_{i\in \mathbb{V}} \A_{ij} C_{ij} \frac{P_j - P_i}{L_{ij}} = S_j \qquad\mbox{for all } j\in \mathbb{V}.
\)

Note that for any given vector of conductivities $C := (C_{ij})_{(i,j)\in \mathbb{E}}$,
\eqref{Kirchhoff0} represents a linear system of equations for the vector of pressures
$(P_j)_{j\in \mathbb{V}}$. The system has a solution, unique up to an additive constant,
if and only if the graph with edge weights given by $C$ is connected \cite{Gross-Yellen},
where only edges with positive conductivities are taken into account (i.e., edges with zero conductivity
are discarded).

\begin{table}
\caption{Notation. (*) denotes variables that are given as data.}
\label{table:notation}
\begin{tabular}{p{2cm}p{5cm}p{3cm}}
\hline\noalign{\smallskip}
Variable & Meaning & Related to \\
%\noalign{\smallskip}\svhline\noalign{\smallskip}
\noalign{\smallskip}\hline\noalign{\smallskip}
$S_j\; (*)$ & intensity of source/sink & vertex $j\in \mathbb{V}$ \\
$P_j$ & pressure & vertex $j\in \mathbb{V}$ \\
$L_{ij}\; (*)$ & length of an edge & edge $(i,j)\in \mathbb{E}$ \\
$Q_{ij}$ & flow from $j\in \mathbb{V}$ to $i\in \mathbb{V}$ & edge $(i,j)\in \mathbb{E}$ \\
$C_{ij}$ & conductivity & edge $(i,j)\in \mathbb{E}$ \\
\noalign{\smallskip}\hline\noalign{\smallskip}
\end{tabular}
\end{table}

Assuming that the material cost for an edge $(i,j)\in \mathbb{E}$ of the network is proportional
to a power $C_{ij}^\gamma$ of its conductivity,
Hu and Cai \cite{Hu-Cai} consider the energy consumption function of the form
\(  \label{E0}
   E[C] := \frac12 \sum_{i\in \mathbb{V}}\sum_{j\in \mathbb{V}} \left( \frac{Q_{ij}^2}{C_{ij}} + \frac{\nu}{\gamma} C_{ij}^\gamma \right) \A_{ij} L_{ij},
\)
where $\nu>0$ is the metabolic coefficient and $Q_{ij}$ is given by \eqref{Q_i},
where the pressure drop $\frac{P_j - P_i}{L_{ij}}$ is determined by \eqref{Kirchhoff0}.
The first part of the energy consumption \eqref{E0} represents the kinetic energy (pumping power)
of the material flow through the vessels, and we shall call it \emph{pumping term} in the sequel.
The second part represents the metabolic cost of maintaining the network
and shall be called \emph{metabolic term}.
For instance, the metabolic cost for a blood vessel is proportional to
its cross-section area \cite{Murray}.
Modeling blood flow by Hagen-Poiseuille's law,
the conductivity of the vessel is proportional to the square of its cross-section area.
This implies $\gamma=1/2$ for blood vessel systems.
For leaf venations, the material cost is proportional to
the number of small tubes, which is proportional to $C_{ij}$, and
the metabolic cost is due to the effective loss of the
photosynthetic power at the area of the venation cells, which
is proportional to $C_{ij}^{1/2}$. Consequently, the effective
value of $\gamma$ typically used in models of leaf venation lies between $1/2$ and $1$, \cite{Hu-Cai}.
Hu and Cai showed that the optimal networks corresponding to minimizers of \eqref{Kirchhoff0}-\eqref{E0}
exhibit a phase transition at $\gamma=1$, with a ``uniform sheet'' (the network is tiled with loops)
for $\gamma>1$ and a ``loopless tree'' for $\gamma<1$, see also \cite{HMR}.
Moreover, they consider the gradient flow of the energy \eqref{E0} constrained by the Kirchhoff law \eqref{Kirchhoff0},
which leads to the ODE system for the conductivities $C_{ij}$,
\[  \label{HC_ODE}
   \tot{C_{ij}}{t} = \left( \frac{Q_{ij}^2}{C_{ij}^2} - \nu C_{ij}^{\gamma-1}\right) L_{ij} \qquad \mbox{for }(i,j)\in \mathbb{E},
\]
coupled to the Kirchhoff law \eqref{Kirchhoff0} through \eqref{Q_i}.
This system represents an adaptation model which dynamically
responds to local information and can naturally incorporate fluctuations in the flow.

This paper focuses on deriving the rigorous continuum limit of the energy functional \eqref{Kirchhoff0}-\eqref{E0}
as the number of nodes of the underlying graph tends to infinity and the edge lengths $L_{ij}$ tend uniformly to zero.
In a general setting with a sequence of unstructured graphs this is a mathematically very challenging task.
In particular, one has to expect that the object obtained in the limit will depend on the structural and statistical properties of the graph sequence
(connectivity, edge directions and density etc.).
Therefore, we consider the particular setting where the graphs correspond to regular equidistant meshes
in 1D and 2D. As we explain in Section \ref{sec:1D}, the energy minimization problem for \eqref{Kirchhoff0}-\eqref{E0} in the
one-dimensional case is in fact trivial, and the form of the limiting functional is obvious.
However, we use this setting as a toy example and carry out the rigorous limit passage anyway.
The reason is that in the 1D setting we avoid most of the technical peculiarities of the two-dimensional case
and we can focus on the essential idea of the method.
Equipped with this insight, we shall turn to the two-dimensional case (Section \ref{sec:2D}), where the graph is
an equidistant rectangular mesh on a square-shaped domain $\Omega$.

In both the 1D and 2D cases, it is necessary to adopt the additional assumption that the conductivities
are a priori bounded away from zero. In particular, we introduce a modification of the system \eqref{Kirchhoff0}-\eqref{E0}
where the conductivities are of the form $r+C_{ij}$, where $r>0$ is a fixed global constant.
The reason is that we need to guarantee the solvability of the Poisson equation \eqref{Poisson2} below,
which will be obtained in the continuum limit.
Moreover, in the 2D case, the additive terms in the energy functional have to be scaled by the square
of the edge length $L_{ij}$. This is due to the fact that we are embedding the inherently one-dimensional edges of the graph
into two spatial dimensions; see \cite[Section 3.2]{HKM} for details.
Thus, we shall work with the energy functional
\( \label{E1}
   E[C] := \frac12 \sum_{i\in \mathbb{V}}\sum_{j\in \mathbb{V}} \left( \frac{Q_{ij}^2}{r+C_{ij}} + \frac{\nu}{\gamma} (r+C_{ij})^\gamma \right) \A_{ij} L_{ij}^d,
\)
where $d=1$, $2$ is the space dimension, coupled to the (properly rescaled) Kirchhoff law
\(  \label{Kirchhoff1}
   \sum_{i\in \mathbb{V}} \A_{ij} (r+C_{ij}) \frac{P_j - P_i}{L_{ij}} = L_j S_j \qquad\mbox{for all } j\in \mathbb{V}
\)
through
\(  \label{Q_i1}
   Q_{ij} = (r+C_{ij}) \frac{P_j-P_i}{L_{ij}},
\)
where $L_j$ are (abstract) weights that scale linearly with the mean edge length;
see \cite[Section 3.1]{HKM} for details about the scaling in \eqref{Kirchhoff1}.
The main benefit of this paper is the rigorous derivation of the limiting energy functional,
which for the two-dimensional case is of the form
\(  \label{E2}
   \E[c] = \int_\Omega \grad p[c] \cdot (rI + c)\grad p[c] + \frac{\nu}{\gamma} \left( |r+c_1|^{\gamma} + |r+c_2|^{\gamma}  \right)\d\mathbf{x},
\)
with $\mathbf{x} = (x,y)\in\R^2$ and where $p[c]\in H^1(\Omega)$ is a weak solution of the Poisson equation
\(   \label{Poisson2}
   - \grad\cdot ((r I + c) \grad p) = S
\)
subject to no-flux boundary conditions on $\partial\Omega$,
where $I$ is the unit matrix and $c$ is the diagonal $2\times 2$-tensor
\(  \label{c}
   c= \begin{pmatrix} c_1 & 0 \\ 0 & c_2 \end{pmatrix}.
\)
{Here, $S\in L^2(\Omega)$ denotes the source/sink term and in analogy to \eqref{assS} we require $\int_\Omega S\d\mathbf{x}=0$.}
The derivation is based on three steps:
\begin{enumerate}
\item
Establish a connection between the discrete solutions of the Kirchhoff law \eqref{Kirchhoff1}
and weak solutions of the Poisson equation \eqref{Poisson2};
see Section \ref{subsec:Reform1D} in 1D and Sections \ref{subsec:FEM-Poisson}, \ref{subsec:Reform2D} in 2D.
\item
Reformulate the discrete energy functional \eqref{E1}
as an integral functional defined on the set of bounded functions;
see Section \ref{subsec:Reform1D} in 1D and Section \ref{subsec:Reform2D} in 2D.
\item
Show that the sequence of integral functionals $\Gamma$-converges to the energy functional \eqref{E2};
see Section \ref{subsec:conv1D} in 1D and Section \ref{subsec:conv2D} in 2D.
See, e.g., \cite{DalMaso, Braides} for details about $\Gamma$-convergence.
\end{enumerate}
%First, reformulate the discrete energy functional \eqref{E1} as an integral functional defined on the set of bounded functions.
%Second, establish a connection between the discrete solutions of the Kirchhoff law \eqref{Kirchhoff1} and weak solutions of the Poisson equation \eqref{Poisson2}.
%Third, show that the sequence of integral functionals $\Gamma$-converges to the energy functional \eqref{E2}; see, e.g., \cite{DalMaso, Braides} for details about $\Gamma$-convergence.
The $\Gamma$-convergence opens the door for constructing global minimizers of \eqref{E2}--\eqref{Poisson2}
as limits of sequences of minimizers of the discrete problem \eqref{E1}--\eqref{Kirchhoff1}.
However, for this we need strong convergence of the minimizers in an appropriate topology.
In agreement with \cite{HMP15, HMPS16, bookchapter} we introduce diffusive terms into the discrete energy functionals,
modeling random fluctuations in the medium (Section \ref{subsec:1Ddiff} for 1D and Section \ref{subsec:diff2D} in 2D).
The diffusive terms provide compactness of the minimizing sequences in a suitable topology
and facilitate the construction of global minimizers of \eqref{E2}--\eqref{Poisson2}.

Let us note that the steepest descent minimization procedure for \eqref{E2}--\eqref{Poisson2} is represented
by the formal $L^2$-gradient flow. This leads to the system of partial differential equations for $c_1=c_1(t,x,y)$,
$c_2=c_2(t,x,y)$,
\(  \label{cPDE1}
  \begin{aligned}
   \partial_t c_1 &=& (\partial_x p)^2 - \nu (r+c_1)^{\gamma-1},\\
   \partial_t c_2 &=& (\partial_y p)^2 - \nu (r+c_2)^{\gamma-1},
   \end{aligned}
\)
subject to homogeneous Dirichlet boundary data and coupled to \eqref{Poisson2} through \eqref{c}.
The existence of weak solutions and their properties are studied in \cite{HKM}.

{
Finally, let us remark that \cite{Hu} proposed a different PDE system,
derived from the discrete model \cite{Hu-Cai} by certain phenomenological considerations
(laws of porous medium flow, see \cite{bookchapter} for details).
The system consists of a parabolic reaction-diffusion equation for the vector-valued conductivity
field, constrained by a Poisson equation for the pressure,
and was studied in the series of papers \cite{HMP15, HMPS16, AAFM, bookchapter}.
However, a rigorous derivation of the model is still lacking;
moreover, no explicit connection to the system \eqref{cPDE1} has been established so far.
}

\section{An auxiliary Lemma}\label{sec:Auxiliary}
\begin{lemma}\label{lem:weak-strong}
Fix $r>0$, a bounded domain $\Omega\subset\R^d$ with $d\geq 1$, {and $S \in L^2(\Omega)$.}
Let $(c^N)_{N\in\N} \subset L^\infty(\Omega)$ be a sequence of nonnegative, essentially bounded functions on $\Omega$,
such that $c^N \to c\in L^2(\Omega)$ in the norm topology of $L^2(\Omega)$.
Let $(p^N)_{N\in\N}\subset H^1(\Omega)$ be a sequence of zero-average
weak solutions of the Poisson equation
\(   \label{auxPoissonN}
   - \grad\cdot ( (r+c^N)\grad p^N) = S
\)
subject to homogeneous Neumann boundary conditions on $\partial\Omega$.
Then $\grad p^N$ converges to $\grad p$ and $\sqrt{c^N}\grad p^N$ converges to $\sqrt{c}\grad p$
strongly in $L^2(\Omega)$,
where $p$ is the zero-average weak solution of
\(   \label{auxPoisson}
   - \grad\cdot ( (r+c)\grad p) = S
\)
subject to homogeneous Neumann boundary conditions on $\partial\Omega$.
In particular, we have
\(  \label{aux:lim}
   \lim_{N\to\infty} \int_\Omega (r+c^N) |\grad p^N|^2 \d \mathbf{x} =
   \int_\Omega (r+c) |\grad p|^2 \d \mathbf{x}.
\)
\end{lemma}

\begin{remark}
Note that we do \emph{not} assume that $(c^N)_{N\in\N}$ is uniformly bounded in $L^\infty(\Omega)$,
nor that $c\in L^\infty(\Omega)$.
\end{remark}

\begin{proof}
Using $p^N$ as a test function in \eqref{auxPoissonN}, due to the nonnegativity of $c^N$,
we have
\(   \label{auxEst1}
  r \Norm{\grad p^N}_{L^2(\Omega)}^2 \leq \int_\Omega (r+c^N) |\grad p^N|^2 \d \mathbf{x} &=& \int_\Omega S p^N \d \mathbf{x} \\
      &\leq& \frac{1}{2\eps} \Norm{S}_{L^2(\Omega)}^2 + \frac{\eps C_P}{2} \Norm{\grad p^N}_{L^2(\Omega)}^2,  \nonumber
\)
where $C_P$ is the Poincar\'{e} constant. With a suitable choice of $\eps>0$ we obtain a uniform estimate
on $p^N$ in $H^1(\Omega)$.
Consequently, there exists a subsequence of $p^N$ that converges weakly in $H^1(\Omega)$ to some $p\in H^1(\Omega)$.
Since $c^N\to c$ strongly in $L^2(\Omega)$, we can pass to the limit in the distributional formulation of \eqref{auxPoissonN}
to obtain
\(  \label{p-distr}
   \int_\Omega (r+c)\grad p \cdot\grad\phi \d \mathbf{x} = \int_\Omega S\phi\d \mathbf{x}\qquad\mbox{for all } \phi\in C_0^\infty(\Omega).
\)
Noting that \eqref{auxEst1} also implies a uniform bound on $\int_\Omega c^N |\grad p^N|^2 \d \mathbf{x}$,
we have due to the weak lower semicontinuity of the $L^2$-norm,
\(  \label{aux:wlsc}
   \int_\Omega (r+c) |\grad p|^2 \d \mathbf{x} \leq \liminf_{N\to\infty} \int_\Omega (r+c^N) |\grad p^N|^2 \d \mathbf{x} < +\infty.
\)
Consequently, we can use $p$ as a test function in \eqref{p-distr} to obtain
\[
   \int_\Omega (r+c) |\grad p|^2 \d \mathbf{x} = \int_\Omega S p \d \mathbf{x}.
\]
Therefore, using $p^N$ as a test function in \eqref{auxPoissonN},
\[
   \lim_{N\to\infty} \int_\Omega (r+c^N) |\grad p^N|^2 \d \mathbf{x} =
   \lim_{N\to\infty} \int_\Omega S p^N \d \mathbf{x} = \int_\Omega S p \d \mathbf{x} =
   \int_\Omega (r+c) |\grad p|^2 \d \mathbf{x},
\]
which gives \eqref{aux:lim} and, further,
\[
   \limsup_{N\to\infty} \int_\Omega |\grad p^N|^2 \d \mathbf{x} &\leq&
   \limsup_{N\to\infty} \int_\Omega (r+c^N) |\grad p^N|^2 \d \mathbf{x} + \limsup_{N\to\infty} \left( - \int_\Omega c^N |\grad p^N|^2 \d \mathbf{x} \right) \\
   &=&
   \int_\Omega (r+c) |\grad p|^2 \d \mathbf{x} - \liminf_{N\to\infty} \int_\Omega c^N |\grad p^N|^2 \d \mathbf{x}.
\]
%Now, noticing that \eqref{auxEst1} also implies a uniform bound on $\sqrt{c^N} |\grad p^N|$ in $L^2(\Omega)$,
%using the weak lower semicontinuity of the $L^2$-norm, we have
Now, using \eqref{aux:wlsc}, we have
\[
   - \liminf_{N\to\infty} \int_\Omega c^N |\grad p^N|^2 \d \mathbf{x} =
   - \liminf_{N\to\infty} \int_\Omega |\sqrt{c^N} \grad p^N|^2 \d \mathbf{x}
    \leq - \int_\Omega |\sqrt{c} \grad p|^2 \d \mathbf{x}.
\]
Therefore,
\[
   \limsup_{N\to\infty} \int_\Omega |\grad p^N|^2 \d \mathbf{x} \leq \int_\Omega |\grad p|^2 \d \mathbf{x}, % \leq  \liminf_{N\to\infty} \int_\Omega |\grad p^N|^2 \d x,
\]
so that $\lim_{N\to\infty} \Norm{\grad p^N}_{L^2(\Omega)} = \Norm{\grad p}_{L^2(\Omega)}$,
which directly implies that (a subsequence of) $p^N$ converges towards $p$ strongly in $H^1(\Omega)$.
\end{proof}

\section{The 1D equidistant setting}\label{sec:1D}
In this section we consider the spatially one-dimensional setting of the discrete network formation problem,
where the graph $(\mathbb{V}, \mathbb{E})$ is given as a mesh on the interval $[0,1]$. Moreover,
for simplicity we consider the equidistant case, where for a fixed $N\in\N$
construct the sequence of meshpoints $x_i$,
\[
   x_i = ih \qquad\mbox{for } i=0,\dots, N, \mbox{ with } h:=1/N.
\]
We identify the meshpoints $x_i$ with the vertices of the graph, i.e.,
we set $\mathbb{V}:=\{x_i;\; i=0,\dots, N\}$.
The segments $(x_{i-1},x_i)$ connecting any two neighboring nodes
are identified with the edges of the graph, i.e., $\mathbb{E}:=\{ (x_{i-1},x_i); \; i=1,\dots, N\}$.
By a slight abuse of notation, we shall write $i\in \mathbb{V}$ instead of $x_i\in \mathbb{V}$ in the sequel,
and similarly $i\in \mathbb{E}$ instead of $(x_{i-1},x_i)\in \mathbb{E}$.
Moreover, we shall use the notation $C:=(C_i)_{i=1}^N$ with
$C_i\geq 0$ the conductivity of the edge $i\in \mathbb{E}$,
$P_i\in\R$ for the pressure in node $i\in \mathbb{V}$ and $S_i^N\in\R$ for the source/sink in node $i\in \mathbb{V}$ {with $\sum_{i=1}^N S_i^N=0$ by \eqref{assS}}.
With this notation we rewrite the energy functional \eqref{E1} as $E^N[C]: \R^N_+ \mapsto \R$,
\(  \label{EC1D}
   E^N[C] := h \sum_{i=1}^N \frac{Q_i^2}{r+C_i} + \frac{\nu}{\gamma} (r+C_{i})^\gamma,
\)
with the fluxes
\(  \label{QN1D}
   Q_i := (r+C_i) \frac{P_{i-1}-P_i}{h},\qquad\mbox{for } i=1,\dots,N.
\)
Note that we orient the fluxes $Q_i$ such that $Q_i>0$ if the material flows from $x_{i-1}$ to $x_i$.
The Kirchhoff law \eqref{Kirchhoff1} is then written in the form
\(  \label{Kirchhoff1D}
   (r+C_i) \frac{P_{i}-P_{i-1}}{h} + (r+C_{i+1}) \frac{P_{i}-P_{i+1}}{h} = h S_i^N \qquad\mbox{for } i=1,\dots,N-1,
\)
while for the terminal nodes we have
\[
   (r+C_1) \frac{P_0-P_1}{h} = h S_0^N,\qquad (r+C_N) \frac{P_{N}-P_{N-1}}{h} = h S_N^N  \,.
\]

Obviously, in the 1D setting the fluxes $Q_i$ are explicitly calculable from the given set of sources/sinks $(S_i)_{i=0}^N$
since the Kirchhoff law \eqref{Kirchhoff1D} is the chain of equations
\[
    Q_1 &=& h  S_0^N,\\
   -Q_i+Q_{i+1} &=& h S_i^N \quad\mbox{for }i=1,\dots, N-1,\\
   - Q_N &=& h S_N^N,
\]
which has the explicit solution
\(  \label{fluxes1D}
   Q_i = h \sum_{j=0}^{i-1} S_j^N\qquad\mbox{for } i=1,\dots,N-1.
\)
Note that due to the assumption of the global mass balance \eqref{assS} %$\sum_{j=0}^N S_j = 0$,
the ``terminal condition'' for $i=N$ is implicitly satisfied,
\(  \label{fluxes1DN}
   - Q_N = - h \sum_{j=0}^{N-1} S_j = h S_N^N.
\)
With the fluxes given by \eqref{fluxes1D}--\eqref{fluxes1DN}, it is trivial to find the global energy minimizer of \eqref{EC1D},
namely, $(r+C_i)^{\gamma+1} = Q_i^2 / \nu$.
It is also easy to prove that the sequence of the functionals \eqref{EC1D} converges as $h=1/N \to 0$
to the continuous functional
\(  \label{E-1D}
   \E[c] := \int_0^1 \frac{q(x)^2}{r+c(x)} + \frac{\nu}{\gamma} (r+c(x))^{\gamma} \d x, %\qquad q(x):=\int_0^x S(\sigma) \d\sigma  
\)
with $q(x):=\int_0^x S(\sigma) \d\sigma$,
in the sense of Riemannian sums if $c$ is a continuous, nonnegative function.
Therefore, the limit passage to continuum description in the one-dimensional case is trivial.
However, we shall use it as a ``training example'' which avoids most of the technical difficulties
of the two-dimensional setting to gain a clear understanding of the main ideas of the method.

Therefore, we shall ignore the explicit formula \eqref{fluxes1D} for the fluxes $Q_i$
and study the limit as $h = 1/N\to 0$ of the sequence of energy functionals \eqref{EC1D}--\eqref{QN1D}, i.e.,
\(  \label{EN1D}
   E^N[C] = h \sum_{i=1}^N (r+C_i) \left(\frac{P_i-P_{i-1}}{h} \right)^2 + \frac{\nu}{\gamma} (r+C_{i})^\gamma,
\)
where the pressures $P_i$ are calculated as a solution of the Kirchhoff law \eqref{Kirchhoff1D}.
Note that since $r+C_i > 0$ for all $i\in \mathbb{V}$, \eqref{Kirchhoff1D} is solvable,
uniquely up to an additive constant.
In the following we shall show that the sequence \eqref{EN1D} converges,
as $h=\frac{1}{N} \to 0$, to the functional \eqref{E-1D} with $q:=(r+c)\partial_x p[c]$, i.e.,
\(  \label{Ec1D}
   \E[c] = \int_0^1 (r+c)(\partial_x p[c])^2 + \frac{\nu}{\gamma} (r+c)^{\gamma} \d x,
\)
where $p[c]\in H^1(0,1)$ is a weak solution of the Poisson equation
\(   \label{Poisson1D}
   - \partial_x ((r+c)\partial_x p) = S
\)
on $(0,1)$, subject to no-flux boundary conditions.
{Here and in the sequel we fix the source/sink term $S\in L^2(0,1)$ and,
in agreement with \eqref{assS}, we assume
the global mass balance $\int_0^1 S(x) \d x = 0$.}
Since for $c(x)\geq 0$ the weak solution $p=p(x)$ of \eqref{Poisson1D} is unique up to an additive constant,
we shall, without loss of generality, always choose the zero-average solution, i.e., $\int_0^1 p(x) \d x = 0$.

We shall proceed in several steps: First, we put the discrete energy functionals \eqref{EN1D}
into an integral form, and find an equivalence between solutions of the Kirchhoff law \eqref{Kirchhoff1D}
and the above Poisson equation with appropriate conductivity.
Then we show the convergence of the sequence of reformulated discrete energy functionals
towards a continuum one as $h=1/N\to 0$.
Finally, we introduce a diffusive term into the energy functional, which will allow us to
construct global minimizers of the continuum energy functional.

\subsection{Reformulation of the discrete energy functional}\label{subsec:Reform1D}
In the first step we reformulate the energy functionals \eqref{EN1D}
such that they are defined on the space $L_+^\infty(0,1)$
of essentially bounded nonnegative functions on $(0,1)$.
For this purpose, we define the sequence of operators
$\Q_0^N: \R^N \to L^\infty(0,1)$ by
\[
   \Q_0^N: (C_i)_{i=1}^N \mapsto c,\qquad\mbox{with } c(x) \equiv C_i \mbox{ for } x \in (x_{i-1},x_i),\; i=1,\dots,N.
\]
I.e., $\Q_0^N$ maps the sequence $(C_i)_{i=1}^N$ onto the bounded function $c=c(x)$,
constant on each interval $(x_{i-1},x_i)$, $i=1,\dots,N$.
Then, we define the functionals $\E^N: L_+^\infty(0,1) \mapsto \R$,
\(  \label{barEN1D}
   \E^N[c] := \int_0^1 (r+c) \left(\Q_0^N[\Delta^h P]\right)^2 + \frac{\nu}{\gamma} (r+c)^{\gamma} \d x,
\)
with
\(  \label{DeltaP}
   (\Delta^h P)_i := \frac{P_{i}-P_{i-1}}{h},\quad i = 1,\dots, N,
\)
and $P = (P_i)_{i=0}^N$ a solution of the Kirchhoff law \eqref{Kirchhoff1D} with the conductivities $C = (C_i)_{i=1}^N$,
\[
   C_i := \frac{1}{h} \int_{x_{i-1}}^{x_i} c(x) \d x,\quad i = 1,\dots, N.
\]
Then, noting that for each $C=(C_i)_{i=1}^N \in\R_+^N$,
\[
   \frac{1}{h} \int_{x_{i-1}}^{x_i} \Q_0^N[C](x) \d x = C_i\qquad\mbox{for all } i=1,\dots, N,
\]
the discrete energy functional \eqref{EN1D} can be written in the integral form as
$E^N[C] = \E^N[\Q_0^N[C]]$. % with $P = (P_i)_{i=0}^N$ as above.

Moreover, we establish a connection between the solutions of the Kirchhoff law \eqref{Kirchhoff1D}
and weak solutions of the Poisson equation \eqref{Poisson1D} with $c = \Q_0^N[C]$:
{
\begin{lemma}
 For any $C=(C_i)_{i=1}^N\in \R^N_+$ and $S\in L^2(0,1)$ with $\int_0^1 S(x) \d x=0$, let  $p=p(x)\in H^1(0,1)$ be a
 weak solution of the Poisson equation \eqref{Poisson1D} with $c = \Q_0^N[C]$, i.e.,
\(  \label{PoissonQ1D}
- \partial_x \left( (r+ \Q_0^N[C])\partial_x p \right) = S,
\)
subject to no-flux boundary conditions on $(0,1)$.  Then,
\(  \label{P_i}
P_i := p(x_i),\qquad i = 0,\dots,N,
\)
is a solution of the Kirchhoff law \eqref{Kirchhoff1D} with the conductivities $C = (C_i)_{i=1}^N$
and the source/sink terms
\(  \label{Sdiscrete}
   S_i^N := \frac{1}{h} \int_0^1 S(x) \phi_i^N(x) \d x, \qquad i=0,\dots,N,
\)
with the hat functions $\phi_i^N = \phi_i^N(x)$ defined in \eqref{phiN} below.
\end{lemma}}

{
\begin{proof}
%, i.e., $\partial_x p(x) = 0$ for $x\in\{0, 1\}$.
Note that for any $C\in\R^N_+$ there exists a weak solution $p=p(x)\in H^1(0,1)$ of \eqref{PoissonQ1D},
unique up to an additive constant.
For $i=1,\dots,N$ we construct the family of piecewise linear test functions $\phi_i^N$,
supported on $(x_{i-1},x_{i+1})$, with
\(  \label{phiN}
   \phi_i^N(x) =\begin{cases} 1 + \frac{x-x_i}{h}\qquad\mbox{for } x\in(x_{i-1},x_i),\\
     1 - \frac{x-x_i}{h}\qquad\mbox{for } x\in(x_i,x_{i+1}).\end{cases}.
\)
Using the hat function $\phi_i^N$ as a test function in \eqref{PoissonQ1D}, we obtain
\[
   (r+C_i)\frac{p(x_{i})-p(x_{i-1})}{h} + (r+C_{i+1})\frac{p(x_{i})-p(x_{i+1})}{h} = h S_i^N,
\]
where we used the fact that, by construction, $\Q_0^N[C] \equiv C_i$
on the interval $(x_{i-1},x_i)$.
Note that due to the embedding $H^1(0,1) \embdd C(0,1)$ any weak solution $p=p(x)$ of \eqref{PoissonQ1D} is a continuous function on $[0,1]$,
so the pointwise values $p(x_i)$ are well defined for all $i=0,\dots,N$.
%Thus, we if $p=p(x)$ is a weak solution of the Poisson equation \eqref{PoissonQ1D}, then defining
Thus, defining  $P_i$ as in \eqref{P_i} 
we obtain a solution of the Kirchhoff law \eqref{Kirchhoff1D} with the conductivities $C = (C_i)_{i=1}^N$
and source/sink terms \eqref{Sdiscrete}.
\end{proof}}

Note that since $\frac1h \int_0^1 \phi_i^N(x) \d x = 1$ and $S\in L^2(0,1)$, the Lebesgue differentiation theorem gives
\[
    S_i^N = \frac{1}{h} \int_0^1 S(x) \phi_i^N(x) \d x \to S(\bar x) \qquad\mbox{for a.e. $\bar x=x_i$ as } h=1/N\to 0.
\]

Consequently, for a fixed $S\in L^2(0,1)$ and any $N\in\N$, we have the following
reformulation of the discrete problem:

\begin{proposition}\label{prop:E1D}
For any vector $C = (C_i)_{i=1}^N \in\R^N_+$, we have
\[
   E^N[C] = \E^N[\Q_0^N[C]],
\]
where $E^N[C]$ is the discrete energy functional \eqref{EN1D} coupled to the Kirchhoff law \eqref{Kirchhoff1D}
with sources/sinks $S_i^N$ given by \eqref{Sdiscrete},
and $\E^N$ is the integral form \eqref{barEN1D}--\eqref{DeltaP} with the pressures given by $P_i = p(x_i)$, $i=0,\dots,N$,
where $p\in H^1(0,1)$ solves the Poisson equation \eqref{PoissonQ1D}.
\end{proposition}

\subsection{Convergence of the energy functionals}\label{subsec:conv1D}
Due to Proposition \ref{prop:E1D}, we are motivated to prove the convergence of the sequence of functionals $\E^N$
given by \eqref{barEN1D}--\eqref{DeltaP} towards {$\E[c]$ given by \eqref{Ec1D}}
%\(  \label{Econt1D}
%   \E[c] := \int_0^1 (r+c) (\partial_x p[c])^2 + \frac{\nu}{\gamma} (r+c)^\gamma \d x,
%\)
with $p[c]\in H^1(0,1)$ a weak solution of \eqref{Poisson1D} with conductivity $c=c(x)$,
equipped with no-flux boundary conditions.
%We set $\E[c] := +\infty$ if \eqref{Poisson1D} is not solvable for the given $c=c(x)$.
We choose to work in the space of essentially bounded functions on $(0,1)$
equipped with the topology of $L^2(0,1)$.
The choice of topology is motivated by the need for strong convergence
of piecewise constant approximations of bounded functions.
Of course, this is true in $L^q(0,1)$ with any $q<+\infty$; our particular choice of $L^2(0,1)$
is further dictated by the fact that we shall apply Lemma \ref{lem:weak-strong} in the sequel.

\begin{lemma}\label{lem:conv1D}
Let $\gamma \geq 0$.
For any sequence of nonnegative functions $(c^N)_{N\in\N}$,
uniformly bounded in $L^\infty(0,1)$ and such that $c^N \to c$ in the norm topology of $L^2(0,1)$ as $N\to\infty$,
we have
%after an eventual extraction of a subsequence,
\[
   \E^N[c^N] \to \E[c] \qquad\mbox{as } h = 1/N\to 0.
\]
\end{lemma}

\begin{proof}
By assumption, $c^N \to c$ in the norm topology of $L^2(0,1)$.
Consequently, there is a subsequence converging almost everywhere on $(0,1)$ to $c$,
and thus $\left(r+ c^N(x) \right)^\gamma$ converges almost everywhere to $(r+c(x))^\gamma$.
Since, by assumption, the sequence $\left(r + c^N(x) \right)^\gamma$ is uniformly bounded in $L^\infty(0,1)$,
we have by the dominated convergence theorem
\[
   \int_0^1 \left(r+ c^N(x) \right)^\gamma \d x \to \int_0^1 (r+ c(x))^\gamma \d x\qquad\mbox{as } h=1/N\to 0.
\]
We recall that the pumping part of the discrete energy $\E^N[c^N]$ \eqref{barEN1D} is of the form
\(   \label{pumping1D}
    \int_0^1 (r+c^N) \left(\Q_0^N[\Delta^h p^N] \right)^2 \d x,
\)
with %$\overline{\Delta^h P} = \Q_0^N[\Delta^h P]$, 
\[
   (\Delta^h p^N)_i := \frac{p^N(x_i)-p^N(x_{i-1})}{h},\quad i = 1,\dots, N,
\]
where $p^N\in H^1(0,1)$ is a solution of the Poisson equation \eqref{Poisson1D} with conductivity $c^N$,
subject to the no-flux boundary condition.
%Obviously, we only need to show that $\left(\Q_0^N[\Delta^h p^N](x) \right)^2$ converges (at least) weakly
%in $L^{q*}(0,1)$, where $q*$ is the adjoint exponent to $q$,
%to $(\partial_x p)^2$ as $N\to \infty$,
%where $p$ is a weak solution of \eqref{Poisson1D} with conductivity $c$.
Let us show that (a subsequence of) %of piecewise constant functions
$\Q_0^N[\Delta^h p^N]$ converges to $\partial_x p[c]$ strongly in $L^2(0,1)$.
We proceed in three steps:
\begin{itemize}
\item
\textbf{Weak convergence.}
By Jensen inequality we have
\(
   \label{Jensen}
   \Norm{\Q_0^N[\Delta^h p^N]}^2_{L^2(0,1)} &=& h \sum_{i=1}^N \left( \frac{p^N(x_i)-p^N(x_{i-1})}{h} \right)^2 \\
   \nonumber
      &=& h \sum_{i=1}^N \left( \frac1h \int_{x_{i-1}}^{x_i} \partial_x p^N(x) \d x \right)^2 
       \leq \int_0^1 (\partial_x p^N)^2 \d x.
\)
Due to the nonnegativity of the functions $c^N$, the right-hand side is uniformly bounded.
Consequently, there exists a weakly converging subsequence of $\Q_0^N[\Delta p^N]$ in $L^2(0,1)$.
\item
\textbf{Identification of the limit.}
For a smooth, compactly supported test function $\psi\in C^\infty_0(0,1)$ we write
\[
   \int_0^1 \Q_0^N[\Delta^h p^N](x) \psi(x) \d x &=& \sum_{i=1}^N \frac{p^N(x_i)-p^N(x_{i-1})}{h} \int_{x_{i-1}}^{x_i} \psi(x) \d x \\
     &=& \frac1h \sum_{i=1}^{N-1} p^N(x_i) \left( \int_{x_{i-1}}^{x_i} \psi(x) \d x - \int_{x_i}^{x_{i+1}} \psi(x) \d x \right)
       + \mbox{``boundary terms''},
\]
where ``boundary terms'' are the two terms with $i=0$ and $i=N$, which we however can neglect for large enough $N$
since $\psi$ has a compact support.

{
Then, Taylor expansion for $\psi$ gives
\[
    \int_{x_{i-1}}^{x_i} \psi(x) \d x - \int_{x_i}^{x_{i+1}} \psi(x) \d x = - h \int_{x_{i-1}}^{x_i}\partial_x \psi(x) \d x
       + \frac{h^2}{2} \int_{x_{i-1}}^{x_i}\partial^2_{xx} \psi(\xi(x)) \d x,
\]
with $\xi(x) \in (x_{i-1}, x_i)$. Due to the estimate
\[
      \left| \frac{h^2}{2} \int_{x_{i-1}}^{x_i}\partial^2_{xx} \psi(\xi(x)) \d x \right| \leq \frac{h^3}{2}\Norm{\partial^2_{xx} \psi}_{L^\infty(0,1)} 
\]
we have
\[
    \int_{x_{i-1}}^{x_i} \psi(x) \d x - \int_{x_i}^{x_{i+1}} \psi(x) \d x = - h \int_{x_{i-1}}^{x_i}\partial_x \psi(x) \d x + \mathcal{O}(h^3),
\]
}
so that
\[
   \int_0^1 \Q_0^N[\Delta^h p^N](x) \psi(x) \d x &=&  - \int_0^1 \overline{p^N} \partial_x \psi(x) \d x + \mathcal{O}(h),
\]
where $\overline{p^N}$ is the piecewise constant function
\[
   \overline{p^N}(x) \equiv p^N(x_i) \quad\mbox{for } x\in(x_{i-1},x_i],\, i=1,\dots, N.
\]
It is easy to check that, due to the strong convergence of $c^N$ towards $c$ in $L^2(0,1)$,
$p^N$ converges to $p[c]$ weakly in $H^1(0,1)$.
Due to the compact embedding $H^1(0,1) \embdd C(0,1)$, (a subsequence of) $p^N$ converges uniformly
to $p[c]$ on $(0,1)$, and, therefore $\overline{p^N}$ converges strongly to $p[c]$.
Therefore,
\[
    \int_0^1 \Q_0^N[\Delta^h p^N](x) \psi(x) \d x &\to& - \int_0^1 p(x) \partial_x \psi(x) \d x \qquad\mbox{as } h=1/N\to 0,\\
    &=& \int_0^1 \psi(x) \partial_x p(x) \d x.
\]
We conclude that weak limit of (the subsequence of) $\Q_0^N[\Delta^h p^N]$ is $\partial_x p[c]$.
\item
\textbf{Strong convergence.}
Finally, due to \eqref{Jensen}, we have
\[
   \Norm{\Q_0^N[\Delta^h p^N] - \partial_x p[c]}^2_{L^2(0,1)} &=&
      \Norm{\Q_0^N[\Delta^h p^N]}^2_{L^2(0,1)} - 2 \langle \Q_0^N[\Delta^h p^N], \partial_x p[c] \rangle_{L^2(0,1)} + \Norm{\partial_x p[c]}^2_{L^2(0,1)} \\
      &\leq& \Norm{\partial_x p^N}^2_{L^2(0,1)} - 2 \langle \Q_0^N[\Delta^h p^N], \partial_x p[c] \rangle_{L^2(0,1)} + \Norm{\partial_x p[c]}^2_{L^2(0,1)},
\]
which vanishes in the limit $h=1/N\to 0$ due to the weak convergence of $\Q_0^N[\Delta p^N]$
and strong convergence of $\partial_x p^N$ in $L^2(0,1)$ due to Lemma \ref{lem:weak-strong}.
Thus, $\Q_0^N[\Delta p^N]$ converges strongly to $\partial_x p[c]$ in $L^2(0,1)$.
\end{itemize}
We conclude that due to the weak-$\ast$ convergence of $(r+c^N)$ towards $(r+c)$ in $L^\infty(0,1)$,
and strong convergence of $\left(\Q_0^N[\Delta^h p^N] \right)^2$ towards $\left(\partial_x p[c] \right)^2$ in $L^1(0,1)$,
we can pass to the limit as $h=1/N\to 0$ in \eqref{pumping1D} to obtain
\[
    \int_0^1 (r+c) \left(\partial_x p[c] \right)^2 \d x.
\]
\end{proof}

%%%%%%%%%%%%%%%%%%%%%%%%%%%%%%%%%%%%%%
\subsection{Introduction of diffusion and construction of continuum energy minimizers}\label{subsec:1Ddiff}
In Lemma \ref{lem:conv1D} we proved the convergence of the sequence of energy functionals
$\E^N$ towards $\E$, i.e., for any $c^N \to c$ in the norm topology of $L^2(0,1)$,
we have $ \E^N[c^N] \to \E[c]$ as $N\to\infty$.
In order to construct energy minimizers of $\E$ as limits of sequences of minimizers of the functionals $\E^N$,
we need to introduce a term into $\E^N$ that shall guarantee compactness of the sequence of discrete minimizers.
This is done, in agreement with \cite{HMP15, HMPS16, bookchapter}, by introducing a diffusive term into the discrete energy functional \eqref{EN1D},
modeling random fluctuations in the medium.
Thus, we construct the sequence $E^N_\mathrm{diff} : \R^N_+ \to \R$,
\(  \label{EN1DwD}
   E^N_\mathrm{diff}[C] := D^2 h \sum_{i=1}^{N-1} \left(\frac{C_{i+1}-C_i}{h}\right)^2 + E^N[C],
\)
with $E^N[C]$ defined in \eqref{EN1D}, coupled to the Kirchhoff law \eqref{Kirchhoff1D}
with sources/sinks $S_i^N$ given by \eqref{Sdiscrete},
and $D^2 > 0$ the diffusion constant.
Note that the new term is a discrete Laplacian acting on the conductivities $C$.

We now need to reformulate the discrete energy functionals \eqref{EN1DwD} in terms of integrals.
For this sake, we construct the sequence of operators $\Q^N_1: \R^N \to C(0,1)$,
where $\Q^N_1[C]$ is a continuous function on $[0,1]$, linear on each interval $(x_i-h/2,x_i+h/2)$, with
\[
      \Q^N_1[C](x_i - h/2) = C_i \quad\mbox{for } i=1,\dots,N,
\]
and
\[
   \Q^N_1[C](x) \equiv C_1 \quad\mbox{for } x\in[0,h/2),\qquad
   \Q^N_1[C](x) \equiv C_N \quad\mbox{for } x\in(1-h/2,1].
\]
Then we write the finite difference term in \eqref{EN1DwD} as
\[
   D^2 h \sum_{i=1}^{N-1} \left(\frac{C_{i+1}-C_i}{h}\right)^2 =
     D^2 \int_0^1 \left(\partial_x \Q^N_1[C] \right)^2 \d x,
\]
and we have

\begin{proposition}\label{prop:E1Ddiff}
For any vector $C = (C_i)_{i=1}^N \in\R^N_+$,
\[
    E^N_\mathrm{diff}[C] = D^2 \int_0^1 \left(\partial_x \Q^N_1[C] \right)^2 \d x + \E^N[\Q_0^N[C]] \,,
\]
where $E^N_\mathrm{diff}$ defined in \eqref{EN1DwD} and
$\E^N$ is given by \eqref{barEN1D}--\eqref{DeltaP} with the pressures given by $P_i = p(x_i)$, $i=0,\dots,N$,
where $p\in H^1(0,1)$ solves the Poisson equation \eqref{PoissonQ1D}.
\end{proposition}

\noindent
We are now ready to prove the main result of this section:

\begin{theorem}%\label{thm:construction}
Let $\gamma\geq 0$, {$S\in L^2(0,1)$ with $\int_0^1 S(x) \d x$ and $S_i^N$ given by \eqref{Sdiscrete}.}
Let $(C^N)_{N\in\N}$ be a sequence of global minimizers of the discrete energy functionals $E^N_\mathrm{diff}$ given by \eqref{EN1DwD}.
Then the sequence $\Q_1^N[C^N]$ converges weakly in $H^1(0,1)$ to $c\in H^1(0,1)$,
a global minimizer of the functional $\E_\mathrm{diff}: H^1_+(0,1) \to \R$,
\[
   %\label{barE1DwD}
   \E_\mathrm{diff}[c] := D^2 \int_0^1 \left(\partial_x c \right)^2 \d x + \E[c],
\]
where $\E[c]$ is given by \eqref{Ec1D}.
\end{theorem}

\begin{proof}
Let us observe that
\[
   E^N_\mathrm{diff}[C^N] \leq E^N_\mathrm{diff}[0] =  {r}  h\sum_{i=1}^{N} \left(\frac{\widetilde P_i- \widetilde P_{i-1}}{h}\right)^2 + \frac{\nu}{\gamma}r^\gamma,
\]
where $(\widetilde P_i)_{i=1}^N$ is a solution of the Kirchhoff law \eqref{Kirchhoff1D} with zero conductivities and sources/sinks given by \eqref{Sdiscrete}.
Thus, $\widetilde P_i = \widetilde p(x_i)$ for $i=1,\dots,N$, where $\widetilde p = \widetilde p(x)$ is a weak solution of $- {r}\laplace p = S$ subject to no-flux boundary conditions.
Then we have by the Jensen inequality
\[
   D^2 h \sum_{i=1}^N \left(\frac{\widetilde P_i- \widetilde P_{i-1}}{h}\right)^2
     &=& D^2 h \sum_{i=1}^{N} \left( \frac{1}{h} \int_{x_{i-1}}^{x_i}  \partial_x \widetilde p \d x \right)^2 \\
     &\leq& D^2 \int_0^1 (\partial_x \widetilde p)^2 \d x.
\]
Consequently, the sequence $\E_\mathrm{diff}^N[C^N]$ is uniformly bounded.

Since the sequence
\[
     D^2 \int_0^1 \left(\partial_x \Q^N_1[C^N] \right)^2 \d x =
    D^2 h \sum_{i=1}^{N-1} \left(\frac{C_{i+1}-C_i}{h}\right)^2 \leq
    E^N_\mathrm{diff}[C^N] 
\]
is uniformly bounded, there exists a subsequence of $\Q^N_1[C^N]$ converging to some $c\in H^1(0,1)$ weakly in $H^1(0,1)$,
and strongly in $L^2(0,1)$; moreover, the sequence is uniformly bounded in $L^\infty(0,1)$.
It is easy to check that also $\Q^N_0[C^N]$ converges to $c$ strongly in $L^2(0,1)$, and is uniformly bounded in $L^\infty(0,1)$.
Therefore, by Lemma \ref{lem:conv1D}, we have ${E^N[C^N] }= \E^N[\Q_0^N[C^N]] \to \E[c]$ as $h=1/N\to 0$.
Moreover, due to the weak lower semicontinuity of the $L^2$-norm, we have
\[
   \int_0^1 \left(\partial_x c \right)^2 \d x \leq  \liminf_{N\to\infty} \int_0^1 \left(\partial_x \Q^N_1[C^N] \right)^2 \d x \,.
\]
Consequently,
\(  \label{gamma_conv_1D}
   \E_\mathrm{diff}[c] \leq  \liminf_{N\to\infty} E^N_\mathrm{diff}[C^N].
\)

We claim that $c$ is a global minimizer of $\E_\mathrm{diff}$ in $H^1_+(0,1)$.
For contradiction, assume that there exists $\bar c\in H^1_+(0,1)$ such that
\[
   \E_\mathrm{diff}[\bar c] < \E_\mathrm{diff}[c].
\]
We define the sequence $(\bar C^N)_{N\in\N}$ by
\[
   \bar C^N_i := \frac{1}{h} \int_{x_{i-1}}^{x_i} \bar c(x) \d x,\quad i = 1,\dots, N.
\]
Then, by assumption, we have for all $N\in\N$,
\(   \label{ass-contr}
   E^N_\mathrm{diff}[\bar C^N]  \geq E^N_\mathrm{diff}[C^N].
\)
It is easy to check that the sequence $\Q_1^N[\bar C^N]$ converges strongly in $H^1(0,1)$ %\blueJH{CHECK!}
towards $\bar c$, therefore
\[
   \int_0^1 \left(\partial_x \Q^N_1[\bar C^N] \right)^2 \d x \to \int_0^1 \left(\partial_x \bar c \right)^2 \d x \qquad\mbox{as } h=1/N\to 0.
\]
Moreover, the sequence $\Q_0^N[\bar C^N]$ converges to $\bar c$ strongly in $L^2(0,1)$, therefore,
by Lemma \ref{lem:conv1D}, $\E^N[\Q_0^N[\bar C^N]] \to \E[\bar c]$ as $h=1/N\to 0$.
Consequently,
\[
   \lim_{h=1/N\to 0} E^N_\mathrm{diff}[\bar C^N] = \E_\mathrm{diff}[\bar c] < \E_\mathrm{diff}[c],
\]
a contradiction to \eqref{gamma_conv_1D}--\eqref{ass-contr}.
\end{proof}

%\begin{remark}\label{rem:gamma1D}
%Let us point out that in the one-dimensional setting the uniform boundedness of the sequence $\Q^N_1[C^N]$ in $H^1(0,1)$,
%implied by the uniform boundedness of the energies $\E^N_\mathrm{diff}[C^N]$,
%implies that $\Q^N_1[C^N]$ and, consequently, $\Q^N_0[C^N]$ are uniformly bounded in $L^\infty(0,1)$
%due to the embedding $H^1(0,1) \embdd L^\infty(0,1)$.
%Thus, we can pass to the limit $h=1/N\to 0$ in the metabolic term $\int_0^1 \left(\Q^N_0[C^N]\right)^\gamma \d x$
%for all $\gamma>0$, see the proof of Lemma \ref{lem:conv1D}.
%Since in the two-dimensional setting it is no longer true that $H^1(0,1) \subset L^\infty(0,1)$,
%we shall be forced to restrict the range of $\gamma$ to $\gamma>1$ and use the weak lower semicontinuity
%of the $L^\gamma$-norm (see proof of Lemma \ref{lem:conv2D} below).
%As we will explain in Remark \ref{rem:gamma2D}, this is not just a technical issue,
%but is fundamental to the energy minimization problem.
%\end{remark}

\def\x{\mathbf{x}}
\def\bfX{\mathbf{X}}
\def\X{\mathbb{X}}
\def\calX{\mathcal{X}}
\def\calU{\mathcal{U}}
\def\calT{\mathcal{T}}
\def\calI{\mathcal{I}}
\def\I{\mathbb{I}}
\def\Edges{\mathscr{E}}
\def\bbW{\mathbb{W}}

%%%%%%%%%%%%%%%%%%%%%%%%%%%%%%%%%%%%%%%%%%%%%%%%%%%%%%%%%%%
\section{The 2D rectangular equidistant setting}\label{sec:2D}
In this section we consider the spatially two-dimensional setting of the discrete network formation problem,
where the graph $(\mathbb{V}, \mathbb{E})$ is embedded in the rectangle $\Omega := [0,1]^2$.
We introduce the notation $\x:=(x,y)\in\Omega$.
For $N\in\N$ we construct the sequence of equidistant rectangular meshes
in $\Omega$ with mesh size $h:=1/N$ and mesh nodes $\bfX_i = (X_i, Y_i)$,
\[
   X_i = (i \mbox{ mod } {N+1})h,\quad Y_i = (i \mbox{ div } {N+1})h,\qquad\mbox{for } i=0,\dots, (N+1)^2-1, \mbox{ with } h:=1/N,
\]
where $(i \mbox{ div } {N+1})$ denotes the integer part of $i/({N+1})$ and $(i \mbox{ mod } {N+1})$ the remainder.
%We denote the set of indices $\{ 0,\dots, (N+1)^2-1 \}$ by $\calI^N$
We identify the mesh nodes $\bfX_i = (X_i, Y_i)$ with the vertices of the graph,
i.e., we set $\mathbb{V}:=\{\bfX_i;\; i = 0,\dots, (N+1)^2-1\}$.
By a slight abuse of notation, we shall write $i\in \mathbb{V}$ instead of $X_i\in \mathbb{V}$ in the sequel.
For each node $\bfX_i$, we denote by $\bfX_{i,E}$, $\bfX_{i,W}$, $\bfX_{i,N}$, $\bfX_{i,S}$
its direct neighbors to the East, West, North and South, respectively (if they exist);
see Fig. \ref{fig:triangles}.
Then, the set $\mathbb{E}$ of edges of the graph is composed of the
horizontal and vertical segments connecting the neighboring nodes,
i.e., $(\bfX_{i}, \bfX_{i,\star})$ for $\star\in\{E,W,N,S\}$ and $i\in \mathbb{V}$.
We shall denote $C_i^\star$ the conductivity of the edge $(\bfX_i, \bfX_{i,\star})$,
and $P_i$, resp., $P_{i,\star}$, denote the pressure in the vertex $\bfX_i$, resp., $\bfX_{i,\star}$.
Similarly, $S_i^h$ denotes the source/sink in vertex $i\in \mathbb{V}$.

%We denote $W_i$ the convex envelope of the nodes $\{ \bfX_{i,E}, \bfX_{i,W}, \bfX_{i,N}, \bfX_{i,S}\}$,
%i.e., the square with these nodes as vertices, for each interior node $\bfX_i$.
%Note that $|W_i| = 2h^2$.
%Moreover, we decompose each $W_i$ into the four triangles 
%$T_i^{NE}$, $T_i^{NW}$, $T_i^{SE}$, $T_i^{SW}$, adjacent to the respective nodes.
%note that $|\calT| = 2N(N+1)$.
%See Figure \ref{fig:mesh}.

With this notation, the discrete energy functional \eqref{E1} takes the particular form
\(  \label{EN2D}
   E^h[C] = \frac{h^2}{2} \sum_{i\in \mathbb{V}} \sum_{\star\in\{E,W,N,S\}} (r+C_i^\star) \left( \frac{P_i - P_{i,\star}}{h} \right)^2 + \frac{\nu}{\gamma}(r+C_i^\star)^\gamma,
\)
and the Kirchhoff law \eqref{Kirchhoff1} is written as
\(  \label{Kirchhoff2D}
   \sum_{\star\in\{E,W,N,S\}} (r+C_i^\star) \frac{P_i - P_{i,\star}}{h} = {h S_i^h}.
   \qquad\mbox{for } i\in \mathbb{V},
\)
%where the meaning of the scaling of the source term ${S^h = (S_i^h)_{i\in V^N}}$ by $h^2$ will become clear later.
For reasons explained later, we shall restrict to the case $\gamma>1$ in the sequel.
%Note that \eqref{Kirchhoff2D} is a linear system of equations for the pressures $(P_i)_{i\in V}$
%and is solvable uniquely up to an additive constant if the graph is connected.

Our strategy is to perform a program analogous to the 1D case of Section \ref{sec:1D}:
first, to put the discrete energy functionals \eqref{EN2D} into an integral form and find an equivalence
between solutions of the Kirchhoff law \eqref{Kirchhoff2D}
and the above Poisson equation with appropriate conductivity.
However, in the 2D case the situation is more complicated and we
need to introduce a finite element discretization of the Poisson equation.
We then establish a connection between the FE-discretization
and the Kirchhoff law \eqref{Kirchhoff2D}.
In the next step we show the convergence of the sequence of reformulated discrete energy functionals
towards a continuum one as $h = 1/N\to 0$, using standard results of the theory of finite elements.
Finally, we introduce a diffusive term into the energy functional, which will allow us to
construct global minimizers of the continuum energy functional.

\subsection{Finite element discretization of the Poisson equation}
\label{subsec:FEM-Poisson}
We construct a regular triangulation on the domain $\Omega$ such that each interior node $\bfX_i$
has six adjacent triangles,
$T_i^{NE}$, $T_i^N$, $T_i^{NW}$, $T_i^{SW}$, $T_i^S$, $T_i^{SE}$,
see Fig. \ref{fig:triangles}.
Boundary nodes have three, two or only one adjacent triangles, depending on their location.
The union of the triangles adjacent to each $\bfX_i$ is denoted by $U_i$.
The collection of all triangles constructed in $\Omega$ is denoted by $\calT^h$.

\begin{figure}
{\centering
\resizebox*{0.5\linewidth}{!}{\includegraphics{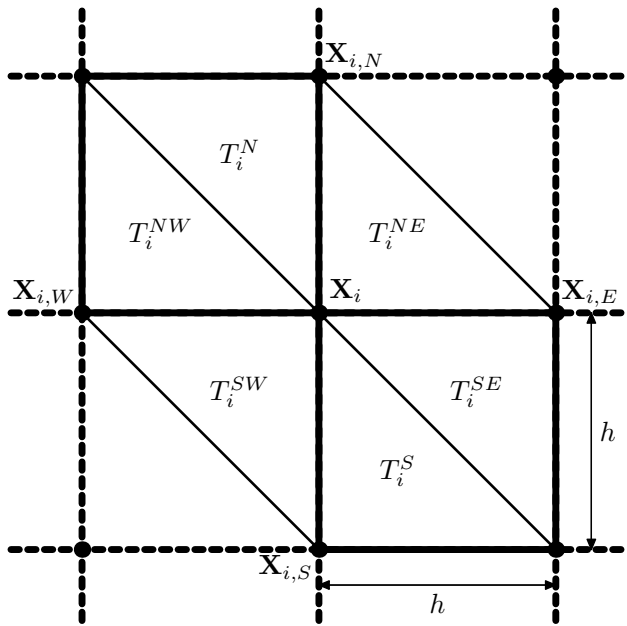}}
\par}
\caption{Interior node $\bfX_i$ with its four neighboring nodes $\bfX_{i,E}$, $\bfX_{i,W}$, $\bfX_{i,N}$, $\bfX_{i,S}$
and six adjacent triangles,
$T_i^{NE}$, $T_i^N$, $T_i^{NW}$, $T_i^{SW}$, $T_i^S$, $T_i^{SE}$.
\label{fig:triangles}}
\end{figure}

{We fix $S\in L^2(\Omega)$ with $\int_\Omega S\d \mathbf{x}=0$ and}
consider a discretization of the Poisson equation
\(  \label{PoissonForFEM}
   - \grad\cdot ((rI + c) \grad p) = S
\)
on $\Omega$ subject to the no-flux boundary conditions,
using the first-order (piecewise linear) $H^1$ finite element method on the triangulation $\calT^h$.
Therefore, on each $NE$-triangle $T_i^{NE}$ we construct the linear basis functions $\phi_{i;1}^{NE}$,
$\phi_{i;2}^{NE}$, $\phi_{i;3}^{NE}$ with
\[
   \phi_{i;1}^{NE}(\bfX_i) = 1, \quad \phi_{i;1}^{NE}(\bfX_{i,E}) = 0, \quad \phi_{i;1}^{NE}(\bfX_{i,N}) = 0,\\
   \phi_{i;2}^{NE}(\bfX_i) = 0, \quad \phi_{i;2}^{NE}(\bfX_{i,E}) = 1, \quad \phi_{i;2}^{NE}(\bfX_{i,N}) = 0,\\
   \phi_{i;3}^{NE}(\bfX_i) = 0, \quad \phi_{i;3}^{NE}(\bfX_{i,E}) = 0, \quad \phi_{i;3}^{NE}(\bfX_{i,N}) = 1,
\]
and analogously for the other triangles in $U_i$, {see Section \ref{subsec:App1} of the Appendix for explicit formulae.}
Denoting $W^h\subset H^1(\Omega)$ the space of continuous, piecewise linear functions on the triangulation $\calT^h$,
the finite element discretization of \eqref{PoissonForFEM} reads
\(  \label{PoissonFEM}
   \int_\Omega  \grad p^h \cdot (rI+c) \grad \psi^h \d \x = \int_\Omega S \psi^h \d \x \qquad\mbox{for all } \psi^h\in W^h.
\)
Using standard arguments (coercivity and continuity of the corresponding bilinear form) we construct
a solution $p^h\in W^h$ of \eqref{PoissonFEM}, unique up to an additive constant;
without loss of generality we fix $\int_\Omega p^h(\x) \d\x = 0$.
The solution is represented by its vertex values $P_i^h := p^h(\bfX_i)$, $i\in \mathbb{V}$.
In particular, on each $NE$-triangle $T_i^{NE}$ we have
\[
    p^h(\x) = P^h_i \phi_{i;1}^{NE}(\x) + P^h_{i,E} \phi_{i;2}^{NE}(\x) + P^h_{i,N} \phi_{i;3}^{NE}(\x), \qquad\x\in T_i^{NE},
\]
and the gradient of $p^h$ on $T_i^{NE}$ is the constant vector
\(  \label{grad_p_h}
   \grad p^h(\x) = \frac{1}{h} ( P^h_{i,E} - P^h_i, P^h_{i,N} - P^h_i), \qquad\x\in T_i^{NE}.
\)
Analogous formulae hold for all other triangles in $U_i$, {as explicitly listed in Section \ref{subsec:App2} of the Appendix.}

We now establish a connection between the discretized Poisson equation \eqref{PoissonFEM} and the Kirchhoff law \eqref{Kirchhoff2D}.
For this purpose, we define the sequence of operators $\Q_0^h$
mapping the vector of conductivities $(C_i)_{i\in \mathbb{E}}$
onto piecewise constant $2\times 2$ diagonal tensors,
\(  \label{Q0N_2D_0}
   \Q_0^h: (C_i)_{i\in \mathbb{E}} \mapsto
     \begin{pmatrix} c_1 & 0 \\ 0 & c_2 \end{pmatrix}.
\)
The functions $c_1=c_1(\x)$, $c_2=c_2(\x)$, defined on $\Omega$, are constant on each triangle $T\in\calT^h$
and $c_1$ takes the value of the conductivity of the horizontal edge of $T$
and $c_2$ takes the value of the conductivity of the vertical edge of $T$.
%For instance, on the NE-triangles we define
%\(  \label{Q0N_2D_1}
%  c_1(\x) :\equiv C_i^E, \quad c_2(\x) :\equiv C_i^N, \qquad\mbox{for } \x\in T_i^{NE},
%\)

{
In particular, we have
\( \label{Q0N_2D_1}
  c_1 := \begin{cases}
     C_i^E\qquad\mbox{on } T_i^{NE}, \\
     C_i^E\qquad\mbox{on } T_i^{SE}, \\
     C_i^W\qquad\mbox{on } T_i^{SW}, \\
     C_i^W\qquad\mbox{on } T_i^{NW},
     \end{cases}
     \qquad
  c_2 := \begin{cases}
     C_i^N\qquad\mbox{on } T_i^{NE}, \\
     C_i^S\qquad\mbox{on } T_i^{S}, \\
     C_i^S\qquad\mbox{on } T_i^{SW}, \\
     C_i^N\qquad\mbox{on } T_i^{N}.
     \end{cases}
\)
}
Then, for a given vector of conductivities $C=(C_i)_{i\in \mathbb{E}}$ we consider
the discretized Poisson equation \eqref{PoissonFEM} with the conductivity tensor $c:=\Q_0^h[C]$.

{
For each $i\in \mathbb{V}$ we construct the test function $\psi^h_i$ as
\[
   \psi^h_i := \phi_{i;1}^{NE} + \phi_{i;1}^{SE} + \phi_{i;1}^{S} + \phi_{i;1}^{SW} + \phi_{i;1}^{NW} + \phi_{i;1}^{N},
\]
with the basis functions on the right-hand side defined in Section \ref{subsec:App1} of the Appendix.
Consequently, each $\psi^h_i$ is supported on $U_i$, linear on each triangle belonging to $U_i$, and
continuous on $\Omega$.
%\[ \psi_i^h(\bfX_i) = 1,\qquad \psi_i^h = 0\mbox{ on }\partial U_i. \]
Then, obviously, $\psi_i^h\in W^h$ and using it as a test function in \eqref{PoissonFEM}, we calculate,
for the triangle $T_i^{NE}$,
\[
   \int_{T_i^{NE}} \grad p^h \cdot \left(r I + \Q_0^h[C]\right) \grad \psi^h_i \d \x = \frac{r+C_{i}^E}{2} \left( P_i^h - P^h_{i,E} \right) + \frac{r+C_{i}^N}{2} \left( P_i^h - P^h_{i,N} \right),
\]
where we used \eqref{grad_p_h}, the identity $\grad \psi^h_i \equiv - \frac{1}{h}(1,1)$ on $T_i^{NE}$,
and orthogonality relations between gradients of the basis functions
(for instance, $\grad \phi_{i;2}^{NE} \cdot \grad \phi_{i;3}^{NE} = 0$).
Performing analogous calculations for the remaining triangles constituting $U_i$, namely, $T_i^{SE}$, $T_i^{S}$, $T_i^{SW}$, $T_i^{NW}$ and $T_i^{N}$,
see Section \ref{subsec:App3} of the Appendix for explicit details, we obtain
\(  \label{CforA}
   \int_\Omega \grad p^h \cdot \left(r I+ \Q_0^h[C]\right) \grad \psi^h_i \d \x &=& \sum_{\star\in\{E,W,N,S\}} (r+C_i^\star) (P_i^h - P_{i,\star}^h).
\)
Consequently, \eqref{PoissonFEM} gives the identity
\[
   \sum_{\star\in\{E,W,N,S\}} (r+C_i^\star) \frac{P_i^h - P_{i,\star}^h}{h} = \frac{1}{h}\int_\Omega S \psi^h_i \d\x
\]
for all $i\in \mathbb{V}$.
}
Thus, defining
\(   \label{S_i^h}
   S_i^h := \frac{1}{h^2} \int_\Omega S \psi^h_i \d \x,
\)
we have the following result:

{
\begin{lemma}\label{lem:KP2D}
For any vector of nonnegative conductivities $C=(C_i)_{i\in \mathbb{E}}$ and $S\in L^2(\Omega)$ with $\int_\Omega S \d \mathbf{x}=0$,
let $p^h\in W^h$ be a solution of the finite element discretization \eqref{PoissonFEM}
with $c := \Q_0^h[C]$.
Then, $P_i^h := p^h(\bfX_i)$, $i\in \mathbb{V}$,
is a solution of the rescaled Kirchhoff law \eqref{Kirchhoff2D} with the source/sink terms $S_i^h$ given by \eqref{S_i^h}.
\end{lemma}}

Note that since $\frac1{h^2} \int_{\Omega} \psi_i^h(\x) \d\x = 1$, and, by assumption, $S\in L^2(\Omega)$,
the Lebesgue differentiation theorem gives
\[
    S_i^h = \frac{1}{h^2} \int_\Omega S \psi^h_i \d \x \to S(\bar\x) \qquad\mbox{for a.e. $\bar\x=\bfX_i$ as } h=1/N\to 0.
\]
Consequently, $(S_i^h)_{h>0}$ is an approximating sequence for the datum $S=S(\x)$.

\subsection{Reformulation of the discrete energy functional}\label{subsec:Reform2D}
We reformulate the energy functionals \eqref{EN2D}--\eqref{Kirchhoff2D}
such that they are defined on the space $L_+^\infty(\Omega)^{2\times 2}_\mathrm{diag}$
of essentially bounded diagonal nonnegative tensors on $\Omega$.
We define the functional $\E^h: L_+^\infty(\Omega)^{2\times 2}_\mathrm{diag} \to \R$,
\(  \label{barEN2D}
   \E^h[c] := \int_\Omega \grad p^h[c] \cdot (rI + c)\grad p^h[c] + \frac{\nu}{\gamma} \left( |r+c_1|^{\gamma} + |r+c_2|^{\gamma}  \right)\d\x,
\)
where $p^h[c]\in W^h$ is a solution of the finite element problem \eqref{PoissonFEM}.
%The symbol $|c|$ denotes the Frobenius norm; $|c_1|^\gamma + |c_2|^\gamma$.

\begin{proposition}\label{prop:reform2D}
Let {$S\in L^2(\Omega)$ with $\int_\Omega S \d \mathbf{x}=0$} and $S_i^h$ be given by \eqref{S_i^h}.
Then for any vector of nonnegative conductivities $C=(C_i)_{i\in E^N}$, we have
\[
   E^h[C] = \E^h[\Q_0^h[C]],
\]
with $E^h$ defined in \eqref{EN2D} and $\E^h$ given by \eqref{barEN2D}.
\end{proposition}

\begin{proof}
We have shown in Section \ref{subsec:FEM-Poisson} that if $p^h=p^h(\x)$ denotes a solution of the finite element problem \eqref{PoissonFEM}
with $c=\Q_0^h[C]$, then the vertex values $P_i^h := p^h(\bfX_i)$ satisfy the Kirchhoff law \eqref{Kirchhoff2D}.
Moreover, using \eqref{grad_p_h} and the definition \eqref{Q0N_2D_0}--\eqref{Q0N_2D_1} of $\Q_0^h[C]$, we calculate
\[
   \int_{T_i^{NE}} \grad p^h \cdot (rI+\Q_0^h[C]) \grad p^h \d\x = |T_i^{NE}| \left( (r+C_i^E) \left(\frac{P_{i,E}^h - P_i^h}{h}\right)^2 + (r+C_i^N) \left(\frac{P_{i,N}^h - P_i^h}{h}\right)^2 \right)
\]
for each $i\in \mathbb{V}$, and analogously for all other triangles.
Noting that $ |T_i^{NE}| = h^2/2$ and summing over all triangles, we obtain the formula \eqref{EN2D} for the discrete energy $E^h[C]$.
\end{proof}

\subsection{Convergence of the energy functional}\label{subsec:conv2D}
With Proposition \ref{prop:reform2D}, our task is now to prove the convergence of the sequence of functionals $\E^h$
given by \eqref{barEN2D} towards
\(  \label{barE2D}
   \E[c] := \int_\Omega \grad p[c] \cdot (rI + c)\grad p[c] + \frac{\nu}{\gamma} \left( |r+c_1|^{\gamma} + |r+c_2|^{\gamma}  \right)\d\x,
\)
where $p[c]\in H^1(\Omega)$ is a weak solution of the Poisson equation \eqref{PoissonForFEM} subject to no-flux boundary conditions,
and $c_1$, $c_2$ are the diagonal entries of $c= \begin{pmatrix} c_1 & 0 \\ 0 & c_2 \end{pmatrix}$.
Similarly as in Section \ref{subsec:conv1D} we choose to work in the space $L_+^\infty(\Omega)^{2\times 2}_\mathrm{diag}$
of diagonal nonnegative tensors on $\Omega$ with essentially bounded entries, equipped with the norm
topology of $L^2(\Omega)$. Note that for $c\in L_+^\infty(\Omega)^{2\times 2}_\mathrm{diag}$ the Poisson equation
\eqref{PoissonForFEM} has a solution $p[c]\in H^1(\Omega)$, unique up to an additive constant,
and $\E[c] < +\infty$.

%For the sake of the proof of the forthcoming Lemma we introduce the square root
%of a diagonal, nonnegative tensor as follows:
%\begin{defin}\label{def:sqrtc}
%For
%\[
 %  c := \begin{pmatrix} c_1 & 0 \\ 0 & c_2 \end{pmatrix}
%\]
%with $c_1$, $c_2 \geq 0$ we define the square root $\sqrt{c}$ by
%\[
 %  \sqrt{c} := \begin{pmatrix} \sqrt{c_1} & 0 \\ 0 & \sqrt{c_2} \end{pmatrix}.
%\]
%\end{defin}

\begin{lemma}
\label{lem:conv2D}
%Let $\gamma > 1$.
For any sequence of nonnegative diagonal tensors $(c^N)_{N\in\N} \subset L_+^\infty(\Omega)^{2\times 2}_\mathrm{diag}$
%such that $\sup_{h=1/N >0} \E[c^N] < +\infty$,
with entries uniformly bounded in $L^\gamma(\Omega)$ and
converging entrywise to $c\in L_+^\gamma(\Omega)^{2\times 2}_\mathrm{diag}$ in the norm topology of $L^2(\Omega)$ as $h=1/N\to 0$,
%and a sequence of meshes with $h:=1/N$
we have, %after an eventual extraction of a subsequence,
\(  \label{conv2D_res}
   \E[c] \leq \liminf_{h=1/N\to 0} \E^h[c^N],
\)
with $\E^h$ given by \eqref{barEN2D} and $\E$ defined in \eqref{barE2D}.
%with $p^h\in W^h(\Omega)$ a solution of the finite element discretization of the Poisson equation \eqref{PoissonFEM} with conductivity $c^N$.
\end{lemma}

\begin{proof}
Due to the strong convergence of the entries of $c^N$ in $L^2(\Omega)$ there exist a subsequence converging almost everywhere
in $\Omega$ to $c$. Then, we have by the Fatou Lemma,
\(  \label{conv_met}
   \int_\Omega |r+c_1|^{\gamma} \d\x \leq \liminf_{h=1/N\to 0} \int_\Omega |r+c^N_1|^{\gamma} \d\x,
\)
which is finite due to the uniform boundedness of $c_1^N$ in $L^\gamma(\Omega)$.
Similarly for $c^N_2$.

For the sequel let us denote $p:=p[c]\in H^1(\Omega)$ is a solution of the Poisson equation \eqref{PoissonForFEM} with conductivity $c$,
$p^N:=p[c^N]$ a solution of the Poisson equation \eqref{PoissonForFEM} with conductivity $c^N$
and $p^h:=p^{h}[c^N]\in W^h$ a solution of the finite element discretization \eqref{PoissonFEM}
with $h=1/N$ and conductivity $c^N$.
Then, by an obvious modification of the auxiliary Lemma \ref{lem:weak-strong} for diagonal tensor-valued conductivities we have
by \eqref{aux:lim},
\(  \label{2Dconv1}
   \int_\Omega \grad p\cdot (rI+c) \grad p \d \x = \lim_{N\to\infty} \int_\Omega \grad p^N \cdot (rI+c^N) \grad p^N \d \x.
\)
Let us define the bilinear forms $B^N: H^1(\Omega)\times H^1(\Omega) \to \R$,
\[
   B^N(u,v) = \int_\Omega \grad u \cdot (rI+c^N) \grad v \d x.
\]
Note that $B^N(u,v) < +\infty$ for $u$, $v\in H^1(\Omega)$ since $c^N \in L_+^\infty(\Omega)^{2\times 2}_\mathrm{diag}$.
Moreover, since $rI+c^N$ is symmetric and positive definite, $B^N$ induces a seminorm on $H^1(\Omega)$,
\[
   \left|u\right|_{B^N} := \sqrt{B^N(u,u)}\qquad\mbox{for } u \in H^1(\Omega).
\]
With this notation we have
\[
   \int_\Omega \grad p^N \cdot (rI+c^N) \grad p^N \d \x = \left|p^N\right|_{B^N}^2.
\]
We now proceed along the lines of standard theory of the finite element method (proof of C\'ea\'s Lemma in the energy norm, see, e.g., \cite{Ciarlet}).
Due to the Galerkin orthogonality
\(  \label{GO}
   B^N(p^N - p^h, \psi) = 0\qquad\mbox{for all }\psi\in W^h,
\)
we have, noting that $p^h\in W^h$,
\[
%   B^N(p^N, p^N) = B^N(p^N-p^h, p^N-p^h) + B^N(p^h, p^h).
   \left|p^N\right|_{B^N}^2 = \left|p^N-p^h\right|_{B^N}^2 + \left|p^h\right|_{B^N}^2.
\]
Then, again by \eqref{GO} and by the Cauchy-Schwartz inequality, we have for all $\psi\in W^h$,
\[
   \left|p^N-p^h\right|_{B^N}^2 = B^N(p^N-p^h, p^N-\psi)
      \leq \left|p^N-p^h\right|_{B^N}  \left| p^N-\psi \right|_{B^N}.
\] 
Therefore, with the triangle inequality,
\[
   \left|p^N-p^h\right|_{B^N} \leq \inf_{\psi\in W^h} \left|p^N-\psi\right|_{B^N}
     \leq \left|p^N-p\right|_{B^N} + \inf_{\psi\in W^h} \left|p-\psi\right|_{B^N}.
\]
Due to the strong convergence of $c^N\to c$ in $L^2(\Omega)$ and 
the standard result of approximation theory, see, e.g., \cite{Ciarlet}, we have
\[
   \lim_{h=1/N\to 0} \; \inf_{\psi\in W^h} \left|p-\psi\right|_{B^N}^2 &\leq&
          \lim_{h\to 0}  \inf_{\psi\in W^h} \int_\Omega \grad (p-\psi)\cdot (rI+c) \grad (p-\psi) \d \x \\
    && +  \lim_{N\to\infty} \int_\Omega \grad p \cdot (c^N-c) \grad p \d \x \\
    &=& 0.
\]
Due to \eqref{2Dconv1} and the weak convergence of $p^N\wto p$ in $H^1(\Omega)$,
\(  \label{2Dconv2}
   \lim_{N\to \infty} \left|p^N-p\right|_{B^N} = 0.
\)
Thus, collecting the above results from \eqref{2Dconv1} up to \eqref{2Dconv2}, we conclude that
\[
   \int_\Omega \grad p\cdot (rI+c) \grad p \d \x = \lim_{N\to\infty} \left|p^N\right|_{B^N}^2 = \lim_{h=1/N\to 0} \left|p^h\right|_{B^N}^2
     = \lim_{h=1/N\to 0} \int_\Omega \grad p^h\cdot (rI+c^N) \grad p^h \d \x,
\]
which together with \eqref{conv_met} gives \eqref{conv2D_res}.
\end{proof}

\begin{remark} \label{rem:conv2D_strong}
Note that if $\gamma>1$ and with the assumption that the sequence $(c^N)_{N\in\N}$ converges (entrywise)
in the norm topology of $L^\gamma(\Omega)$, the statement of Lemma \ref{lem:conv2D}
can be strengthened to
\[  %\label{conv2D_strong}
   \E[c] = \lim_{h=1/N\to 0} \E^h[c^N].
\]
This follows directly from the fact that in this case we have for the metabolic term
\[
   \int_\Omega |r+c_1|^{\gamma} + |r+c_2|^{\gamma} \d\x = \lim_{h=1/N\to 0} \int_\Omega |r+c^N_1|^{\gamma} + |r+c^N_2|^{\gamma} \d\x.
\]
\end{remark}

\noindent
Lemma \ref{lem:conv2D} and Remark \ref{rem:conv2D_strong} trivially imply
the $\Gamma$-convergence of the sequence of energy functionals $\E^h$
in the norm topology of $L^\gamma(\Omega)$ for $\gamma>1$:

\begin{theorem}
\label{thm:GammaConv2D}
Let $\gamma > 1$, {$S\in L^2(\Omega)$ with $\int_\Omega S \d \mathbf{x}=0$} and $S_i^h$ be given by \eqref{S_i^h}.
Then the sequence $\E^h$ given by \eqref{barEN2D} $\Gamma$-converges
to $\E$ defined in \eqref{barE2D} with respect to the norm topology of $L^\gamma(\Omega)$
on the set $L_+^\infty(\Omega)^{2\times 2}_\mathrm{diag}$.
In particular:
\begin{itemize}
\item
For any sequence $(c^N)_{N\in\N} \subset L_+^\infty(\Omega)^{2\times 2}_\mathrm{diag}$
converging entrywise to $c\in L_+^\gamma(\Omega)^{2\times 2}_\mathrm{diag}$ in the norm topology of $L^\gamma(\Omega)$ as $h=1/N\to 0$,
we have
\[  %\label{conv2D_res}
   \E[c] \leq \liminf_{h=1/N\to 0} \E^h[c^N].
\]

\item
For any $c\in L_+^\infty(\Omega)^{2\times 2}_\mathrm{diag}$ there exists a sequence $(c^N)_{N\in\N} \subset L_+^\infty(\Omega)^{2\times 2}_\mathrm{diag}$
converging entrywise to $c\in L_+^\gamma(\Omega)^{2\times 2}_\mathrm{diag}$ in the norm topology of $L^\gamma(\Omega)$ as $h=1/N\to 0$,
such that
\[   %\label{Gamma2D_limsup}
   \E[c] \geq \limsup_{h=1/N\to 0} \E^h[c^N].
\]
\end{itemize}
\end{theorem}

\begin{proof}
The $\liminf$-statement follows directly from Lemma \ref{lem:conv2D}.
For the $\limsup$-statement it is sufficient to set $c^N:=c$ for all $N\in\N$ and use Remark \ref{rem:conv2D_strong},
which in fact leads to the stronger statement
\[   %\label{Gamma2D_limsup}
   \E[c] = \lim_{h=1/N\to 0} \E^h[c^N].
\]
\end{proof}

\def\bfM{\mathbf{M}}
\def\D{\mathbb{D}}

%%%%%%%%%%%%%%%%%%%%%%%%%%%%%%%%%%%%%%
\subsection{Introduction of diffusion and construction of continuum energy minimizers ($\gamma>1$)}\label{subsec:diff2D}
As in the one-dimensional case, we introduce a diffusive term into the discrete energy functionals,
which shall provide compactness of the sequence of energy minimizers.
We again construct a piecewise linear approximation of the discrete conductivities $C$,
which, however, turns out to be technically quite involved in the two-dimensional situation.

We shall describe the process for the conductivities of the horizontal edges, 
and by a slight abuse of notation, we denote $C_{i+1/2,j}$ the conductivity
of the horizontal edge connecting the node $(ih, jh)$ to $((i+1)h,jh)$ {for $i=0,\ldots, N-1, j=0,\ldots,N$ where $h=1/N$}.
Moreover, we denote $\bfM_{i+1/2,j}$ the midpoint of this edge, i.e., $\bfM_{i+1/2,j} = ((i+1/2)h, jh)$.
For a given vector of conductivities $C$, we {construct the continuous function $\Q_1^h[C]$ on $\Omega$, 
such that
\[
   \Q_1^h[C](\bfM_{i+1/2,,j}) = C_{i+1/2,j},\qquad\mbox{for } i = 0,\dots,N-1,\; j=0,\dots,N,
\]
and $\Q_1^h[C]$ is linear on each triangle spanned by the nodes
$\bfM_{i-1/2,j}$, $\bfM_{i+1/2,j}$, $\bfM_{i-1/2,j+1}$ and
on each triangle spanned by the nodes $\bfM_{i+1/2,j}$, $\bfM_{i+1/2,j+1}$, $\bfM_{i-1/2,j+1}$,
for $i=1,\ldots,N-1,\; j=0,\ldots,N-1$.
%, i.e., we consider 
%\begin{align*}
%\Q_1^h[C](\mathbf{x})&=\frac{C_{i+1,j}-C_{ij}}{h}(x_1-(i+1/2)h)+\frac{C_{i,j+1}-C_{ij}}{h}(x_2-jh)\\&\qquad+\frac{C_{i+1,j+1}+C_{ij}-C_{i+1,j}-C_{i,j+1}}{h^2}(x_1-(i+1/2)h)(x_2-jh)+C_{ij}
%\end{align*}
%for $\mathbf{x}=(x_1,x_2)\in A_{ij}$,  $i=0,\ldots,N-2,\; j=0,\ldots,N-1$, implying
Let us denote the union of such two triangles, i.e., the square spanned by the nodes
$\bfM_{i-1/2,j}$, $\bfM_{i+1/2,j}$, $\bfM_{i-1/2,j+1}$ and $\bfM_{i+1/2,j+1}$,
by $W_{ij}$. Then, a simple calculation reveals that
\(  \label{intW}
	\int_{W_{ij}}|\nabla \Q_1^h[C]|^2\d \mathbf{x} 
	   &=& \frac12 \left[ (C_{i-1/2,j}-C_{i+1/2,j})^2 + (C_{i+1/2,j}-C_{i+1/2,j+1})^2\right. \\ && \left. + (C_{i+1/2,j+1}-C_{i-1/2,j+1})^2 + (C_{i-1/2,j+1}-C_{i-1/2,j})^2 \right].
	\nonumber
\)
}
On the ``boundary stripe'' $(0,h/2)\times(0,1)$ and $(1-h/2,1)\times (0,1)$
the function is defined to be constant in the $x$-direction, such that it is
globally continuous on $\Omega$, {i.e., 
\begin{align*}
	\Q_1^h[C](\mathbf{x})&:=\frac{C_{1/2,j+1}-C_{1/2,j}}{h} (x_2-jh)+C_{1/2,j},\qquad \mbox{for }\mathbf{x}=(x_1,x_2)\in (0,h/2)\times (jh,(j+1)h),\\
	\Q_1^h[C](\mathbf{x})&:=\frac{C_{N-1/2,j+1}-C_{N-1/2,j}}{h}(x_2-jh)+C_{N-1/2,j},\qquad \mbox{for }\mathbf{x}=(x_1,x_2)\in (1-h/2,1)\times (jh,(j+1)h)
\end{align*}
for $j=0,\ldots,N-1$.}
Summing up \eqref{intW} over all squares $W_{ij}$ and the boundary stripe, we arrive at
\(  \label{D_x}
   \int_\Omega |\grad \Q_1^h[C]|^2 \d\x = \D_x[C],
\)
with
%   \D_x[C] := \sum_{i=0}^{N-1} \sum_{j=0}^{N-1} &&\hspace{-5mm} (C_{i-1/2,j}-C_{i+1/2,j})^2 + (C_{i+1/2,j}-C_{i+1/2,j+1})^2 \\ && +\; (C_{i+1/2,j+1}-C_{i-1/2,j+1})^2 + (C_{i-1/2,j+1}-C_{i-1/2,j})^2.
{
	\[
	\D_x[C] := &&\hspace{-5mm} \sum_{i=0}^{N-1} \sum_{j=0}^{N-1} (C_{i+1/2,j}-C_{i+1/2,j+1})^2 + \sum_{i=1}^{N-1} \sum_{j=1}^{N-1} (C_{i-1/2,j}-C_{i+1/2,j})^2 \\
	     && + \frac12 \sum_{i=1}^{N-1} \left[ (C_{i-1/2,0}-C_{i+1/2,0})^2 + (C_{i-1/2,N}-C_{i+1/2,N})^2 \right].
	\]
}
%\[
%   \D_x[C] := \sum_{i=0}^{N-1} \sum_{j=0}^{N-1} && \frac12 \left(\frac{C_{i,j}-C_{i+1,j}}{h}\right)^2
%      + \frac12 \left(\frac{C_{i+1,j}-C_{i+1,j+1}}{h}\right)^2  \\
%       &+& \frac12 \left(\frac{C_{i+1,j+1}-C_{i,j+1}}{h}\right)^2 + \frac12 \left(\frac{C_{i,j+1}-C_{i,j}}{h}\right)^2  \\
%       &+& \left(\frac{C_{i,j}-C_{i+1,j+1}}{h}\right)^2 + \left(\frac{C_{i+1,j}-C_{i,j+1}}{h}\right)^2.
%\]
Performing the same procedure for the vertical edges, we obtain
\(  \label{D_y}
   \int_\Omega |\grad \Q_2^h[C]|^2 \d\x = \D_y[C],
\)
with obvious definitions of $ \Q_2^h[C]$ and $\D_y[C]$.

Consequently, we define the sequence of discrete energy functionals $E^h_\mathrm{diff}$,
\(  \label{EN2DwD}
   E^h_\mathrm{diff}[C] := D^2 \left( \D_x[C] + \D_y[C] \right) + E^h[C],
\)
with $D^2>0$ diffusion constant and $E^h[C]$ defined in \eqref{EN2D}, coupled to the Kirchhoff law \eqref{Kirchhoff2D}
with sources/sinks $ S_i^h$ given by \eqref{S_i^h}.
We then have:

\begin{proposition}\label{prop:p2D2}
For any vector ${C = (C_i)_{i\in \mathbb{E}}}$ of nonnegative entries, we have
\[
    E^h_\mathrm{diff}[C] = D^2 \int_\Omega |\grad \Q_1^h[C]|^2 + |\grad \Q_2^h[C]|^2 \d\x + \E^h[\Q_0^h[C]],
\]
with $E^h_\mathrm{diff}$ defined in \eqref{EN2DwD} and
$\E^h$ given by \eqref{barEN2D} with the pressures $p^h$
being a solution of the FEM-discretized Poisson equation \eqref{PoissonFEM} with $c = \Q_0^N[C]$.
\end{proposition}

\noindent
We are now in shape to prove the main result of this section:

\begin{theorem}%\label{thm:construction}
Let $\gamma > 1$, {$S\in L^2(\Omega)$ with $\int_\Omega S \d \mathbf{x}=0$} and $S_i^h$ be given by \eqref{S_i^h}.
Let $(C^N)_{N\in\N}\subset \R^N$ be a sequence of global minimizers of the discrete energy functionals $E^h_\mathrm{diff}$ given by \eqref{EN2DwD}
with $h=1/N$.
Then the sequence of diagonal $2\times 2$ matrices
\[
   c^N := \begin{pmatrix} \Q_1^h[C^N] & 0 \\ 0 & \Q_2^h[C^N] \end{pmatrix}
\]
converges weakly in $H^1(\Omega)^{2\times 2}$ to $c\in H^1(\Omega)^{2\times 2}_+$ as $h=1/N\to 0$,
with $c$ a global minimizer of the functional $\E_\mathrm{diff}: H^1_+(\Omega)^{2\times 2}_\mathrm{diag} \to \R$,
\[
   %\label{barE1DwD}
   \E_\mathrm{diff}[c] := D^2 \int_\Omega |\grad c_1|^2 + |\grad c_2|^2 \d\x + \E[c],
\]
where $\E[c]$ is given by \eqref{barE2D}.
\end{theorem}

\begin{proof}
Let us observe that
\[
   E^h_\mathrm{diff}[C^N] \leq E^h_\mathrm{diff}[0] = 
      \frac{h^2}{2} \sum_{i\in \mathbb{V}} \sum_{\star\in\{E,W,N,S\}} r \left( \frac{\widetilde P_i - \widetilde P_{i,\star}}{h} \right)^2 +\frac{\nu}{\gamma} r^\gamma,
\]
where $(\widetilde P_i)_{i\in \mathbb{V}}$ is a solution of the Kirchhoff law \eqref{Kirchhoff2D} with conductivities $C=0$ and sources/sinks given by \eqref{S_i^h}.
As shown in Section \ref{subsec:FEM-Poisson}, the pressures $\widetilde P_i$ correspond to pointwise values $\widetilde P_i^h := \tilde p^h(\bfX_i)$, $i\in \mathbb{V}$,
of the solution $\tilde p^h$ of the discretized Poisson equation \eqref{PoissonFEM} with conductivity tensor $c=0$.
Moreover, due to formula \eqref{grad_p_h} we have
\[
   \frac{h^2}{2} \sum_{i\in \mathbb{V}} \sum_{\star\in\{E,W,N,S\}} r \left( \frac{\widetilde P_i - \widetilde P_{i,\star}}{h} \right)^2
      =r \int_\Omega |\grad \tilde p^h|^2 \d\x,
\]
and the uniform boundedness of $\grad \tilde p^h$ in $L^2(\Omega)$ implies a uniform bound on $E^h_\mathrm{diff}[C^N]$.

Since the sequence
\[
     D^2 \int_\Omega |\grad \Q_1^h[C^N]|^2 + |\grad \Q_2^h[C^N]|^2 \d\x
     = D^2 \left( \D_x[C^N] + \D_y[C^N] \right)
    \leq E^h_\mathrm{diff}[C^N] 
\]
is uniformly bounded, there exist subsequences of $\Q^h_1[C^N]$ and $\Q^h_2[C^N]$ converging to some $c_1$, $c_2 \in H^1(\Omega)$ weakly in $H^1(\Omega)$,
and strongly in $L^2(\Omega)$.
It is easy to check that then also $\Q^h_0[C^N]$ converges to $c: = \begin{pmatrix} c_1 & 0 \\ 0 & c_2 \end{pmatrix}$ strongly in $L^2(0,1)^{2\times 2}$.
Clearly, we also have $\Q^h_0[C^N]\in L^\infty_+(\Omega)^{2\times 2}_\mathrm{diag}$ with entries uniformly bounded in $L^\gamma(\Omega)$.
Consequently, by Lemma \ref{lem:conv2D}, we have 
\[
   \E[c] \leq \liminf_{h=1/N\to 0} \E^h[c^N].
\]
Moreover, due to the weak lower semicontinuity of the $L^2$-norm, we have
\[
   \int_\Omega |\grad c_1|^2 + |\grad c_2|^2 \d\x  \leq  \liminf_{h=1/N\to 0}  \int_\Omega |\grad \Q_1^h[C^N]|^2 + |\grad \Q_2^h[C^N]|^2 \d\x \,.
\]
Consequently,
\(  \label{gamma_conv_2D}
   \E_\mathrm{diff}[c] \leq   \liminf_{h=1/N\to 0} E^h_\mathrm{diff}[C^N].
\)

We claim that $c$ is a global minimizer of $\E_\mathrm{diff}$ in $H^1_+(\Omega)^{2\times 2}_\mathrm{diag}$.
For contradiction, assume that there exists $\bar c\in H^1_+(\Omega)^{2\times 2}_\mathrm{diag}$ such that
\[
   \E_\mathrm{diff}[\bar c] < \E_\mathrm{diff}[c].
\]
We define the sequence $(\bar C^N)_{N\in\N}$ by setting the conductivity $\bar C^N_i$ of each horizontal edge
$i\in \mathbb{E}$ to the average of $\bar c_1$ over the two triangles $T_{i;1}$, $T_{i;2}\in\calT^h$ that contain the edge $i$, i.e.,
\[
   \bar C^N_i := \frac{1}{h^2} \int_{T_{i;1} \cup T_{i;2}} \bar c_1(x) \d\x.
\]
Similarly, we use the averages of $\bar c_2$ to define the conductivities of the vertical edges.
Then, by assumption, we have for all $h=1/N$, $N\in\N$,
\(   \label{ass-contr2D}
   E^h_\mathrm{diff}[\bar C^N]  \geq E^h_\mathrm{diff}[C^N].
\)
It is easy to check that the sequence $\Q_1^h[\bar C^N]$ converges strongly in $H^1(\Omega)$ %\blueJH{CHECK!}
towards $\bar c_1$, therefore
\[
   \int_\Omega |\grad \Q_1^h[\bar C^N]|^2 \d\x \to  \int_\Omega |\grad \bar c_1|^2 \d\x \qquad\mbox{as } h=1/N\to 0,
\]
and analogously for $\Q_2^h[\bar C^N]$ and $\bar c_2$.
Moreover, the sequence $\Q_0^h[\bar C^N]$ converges to $\bar c$ strongly in $L^\gamma(\Omega)^{2\times 2}_\mathrm{diag}$, therefore,
by Remark \ref{rem:conv2D_strong}, $\E^h[\Q_0^N[\bar C^N]] \to \E[\bar c]$ as $h=1/N\to 0$.
Consequently,
\[
   \lim_{h=1/N\to 0}  E^h_\mathrm{diff}[\bar C^N]  = \E_\mathrm{diff}[\bar c] < \E_\mathrm{diff}[c],
\]
a contradiction to \eqref{gamma_conv_2D}--\eqref{ass-contr2D}.
\end{proof}

\begin{remark}\label{rem:paralellogram}
We can easily generalize to the situation when the two-dimensional grid
is not rectangular, but consists of parallelograms with sides of equal length
in linearly independent directions $\theta_1$, $\theta_2\in\mathbb{S}^1$,
where $\mathbb{S}^1$ is the unit circle in $\R^2$.
Then the coordinate transform
\[
   (1,0) \mapsto \theta_1,\qquad (0,1)\mapsto \theta_2
\]
in \eqref{barE2D} leads to the transformed continuum energy functional
\[
   \E[c] = \int_\Omega \grad p[c]\cdot \mathbb{P}[c]\grad p[c] + \frac{\nu}{\gamma} \left( \left|r+c_1\right|^\gamma + \left|r+c_2\right|^\gamma \right) \d\x
\]
coupled to the Poisson equation
\[
   -\grad\cdot \left(\mathbb{P}[c]\grad p \right) = S
\]
with the permeability tensor
\[
   \mathbb{P}[c] = rI + c_1\theta_1\otimes\theta_1 + c_2\theta_2\otimes\theta_2.
\]
The eigenvalues of $\mathbb{P}[c]$ (principal permeabilities) are
\[
   \lambda_{1,2} = \frac{1}{2} \left( c_1 + c_2 \pm \sqrt{(c_1-c_2)^2 - 4c_1c_2 (\theta_1\cdot\theta_2)^2} \right)
\]
and the corresponding eigenvectors (principal directions)
\[
   u_{1,2} = \theta_1 + \frac{c_2 - c_1 \pm \sqrt{(c_1-c_2)^2 - 4c_1c_2 (\theta_1\cdot\theta_2)^2}}{2 c_1 \theta_1\cdot\theta_2}\theta_2.
\]
\end{remark}

%%%%%%%%%%%%%%%%%%%%%%%%%%%%%%%%%
%% Appendix
%%%%%%%%%%%%%%%%%%%%%%%%%%%%%%%%%
\section{Appendix}\label{sec:App}
Here we provide more technical details for the constructions and calculations
performed in Section~\ref{subsec:FEM-Poisson}.

%%%%%%%%%%%%%%%%%%%%%%%%%%%%%%%%%
\subsection{Linear basis functions}\label{subsec:App1}
We list the explicit definitions for the piecewise linear basis functions
%of the first-order $H^1$ finite element method
on the triangulation $\calT^h$,
constructed in Section \ref{subsec:FEM-Poisson}.
Any interior node $i\in \mathbb{V}$ has six adjacent triangles, denoted clockwise by $T_i^{NE}$, $T_i^{SE}$, $T_i^{S}$, $T_i^{SW}$, $T_i^{NW}$, $T_i^{N}$,
see Fig. \ref{fig:triangles}.
For each triangle we construct three basis functions, supported on the respective triangle and linear on their support.
Obviously, the basis functions are uniquely determined by their values on the triangle vertices.
For later reference we list their gradients, which are constant vectors on the respective triangles.

\begin{itemize}
\item
On the $NE$-triangle $T_i^{NE}$ we construct the linear basis functions $\phi_{i;1}^{NE}$,
$\phi_{i;2}^{NE}$, $\phi_{i;3}^{NE}$ defined by
\[
   \phi_{i;1}^{NE}(\bfX_i) = 1, \quad \phi_{i;1}^{NE}(\bfX_{i,E}) = 0, \quad \phi_{i;1}^{NE}(\bfX_{i,N}) = 0,\\
   \phi_{i;2}^{NE}(\bfX_i) = 0, \quad \phi_{i;2}^{NE}(\bfX_{i,E}) = 1, \quad \phi_{i;2}^{NE}(\bfX_{i,N}) = 0,\\
   \phi_{i;3}^{NE}(\bfX_i) = 0, \quad \phi_{i;3}^{NE}(\bfX_{i,E}) = 0, \quad \phi_{i;3}^{NE}(\bfX_{i,N}) = 1,
\]
so that
\[
   \grad \phi_{i;1}^{NE} \equiv -\frac{1}{h}(1,1), \qquad
   \grad \phi_{i;2}^{NE} \equiv \frac{1}{h}(1,0), \qquad
   \grad \phi_{i;3}^{NE} \equiv \frac{1}{h}(0,1), \qquad
   \mbox{on } T_i^{NE}.
\]
\item
On the $SE$-triangle $T_i^{SE}$ we construct the linear basis functions $\phi_{i;1}^{SE}$,
$\phi_{i;2}^{SE}$, $\phi_{i;3}^{SE}$ defined by
\[
   \phi_{i;1}^{SE}(\bfX_i) = 1, \quad \phi_{i;1}^{SE}(\bfX_{i,E}) = 0, \quad \phi_{i;1}^{SE}(\bfX_{i,SE}) = 0,\\
   \phi_{i;2}^{SE}(\bfX_i) = 0, \quad \phi_{i;2}^{SE}(\bfX_{i,E}) = 1, \quad \phi_{i;2}^{SE}(\bfX_{i,SE}) = 0,\\
   \phi_{i;3}^{SE}(\bfX_i) = 0, \quad \phi_{i;3}^{SE}(\bfX_{i,E}) = 0, \quad \phi_{i;3}^{SE}(\bfX_{i,SE}) = 1,
\]
so that
\[
   \grad \phi_{i;1}^{SE} \equiv -\frac{1}{h}(1,0), \qquad
   \grad \phi_{i;2}^{SE} \equiv \frac{1}{h}(1,1), \qquad
   \grad \phi_{i;3}^{SE} \equiv -\frac{1}{h}(0,1), \qquad
   \mbox{on } T_i^{SE}.
\]
\item
On the $S$-triangle $T_i^{S}$ we construct the linear basis functions $\phi_{i;1}^{S}$,
$\phi_{i;2}^{S}$, $\phi_{i;3}^{S}$ defined by
\[
   \phi_{i;1}^{S}(\bfX_i) = 1, \quad \phi_{i;1}^{S}(\bfX_{i,SE}) = 0, \quad \phi_{i;1}^{S}(\bfX_{i,S}) = 0,\\
   \phi_{i;2}^{S}(\bfX_i) = 0, \quad \phi_{i;2}^{S}(\bfX_{i,SE}) = 1, \quad \phi_{i;2}^{S}(\bfX_{i,S}) = 0,\\
   \phi_{i;3}^{S}(\bfX_i) = 0, \quad \phi_{i;3}^{S}(\bfX_{i,SE}) = 0, \quad \phi_{i;3}^{S}(\bfX_{i,S}) = 1,
\]
so that
\[
   \grad \phi_{i;1}^{S} \equiv \frac{1}{h}(0,1), \qquad
   \grad \phi_{i;2}^{S} \equiv \frac{1}{h}(1,0), \qquad
   \grad \phi_{i;3}^{S} \equiv -\frac{1}{h}(1,1), \qquad
   \mbox{on } T_i^{S}.
\]
\item
On the $SW$-triangle $T_i^{SW}$ we construct the linear basis functions $\phi_{i;1}^{SW}$,
$\phi_{i;2}^{SW}$, $\phi_{i;3}^{SW}$ defined by
\[
   \phi_{i;1}^{SW}(\bfX_i) = 1, \quad \phi_{i;1}^{SW}(\bfX_{i,S}) = 0, \quad \phi_{i;1}^{SW}(\bfX_{i,W}) = 0,\\
   \phi_{i;2}^{SW}(\bfX_i) = 0, \quad \phi_{i;2}^{SW}(\bfX_{i,S}) = 1, \quad \phi_{i;2}^{SW}(\bfX_{i,W}) = 0,\\
   \phi_{i;3}^{SW}(\bfX_i) = 0, \quad \phi_{i;3}^{SW}(\bfX_{i,S}) = 0, \quad \phi_{i;3}^{SW}(\bfX_{i,W}) = 1,
\]
so that
\[
   \grad \phi_{i;1}^{SW} \equiv \frac{1}{h}(1,1), \qquad
   \grad \phi_{i;2}^{SW} \equiv -\frac{1}{h}(0,1), \qquad
   \grad \phi_{i;3}^{SW} \equiv -\frac{1}{h}(1,0), \qquad
   \mbox{on } T_i^{SW}.
\]
\item
On the $NW$-triangle $T_i^{NW}$ we construct the linear basis functions $\phi_{i;1}^{NW}$,
$\phi_{i;2}^{NW}$, $\phi_{i;3}^{NW}$ defined by
\[
   \phi_{i;1}^{NW}(\bfX_i) = 1, \quad \phi_{i;1}^{NW}(\bfX_{i,W}) = 0, \quad \phi_{i;1}^{NW}(\bfX_{i,NW}) = 0,\\
   \phi_{i;2}^{NW}(\bfX_i) = 0, \quad \phi_{i;2}^{NW}(\bfX_{i,W}) = 1, \quad \phi_{i;2}^{NW}(\bfX_{i,NW}) = 0,\\
   \phi_{i;3}^{NW}(\bfX_i) = 0, \quad \phi_{i;3}^{NW}(\bfX_{i,W}) = 0, \quad \phi_{i;3}^{NW}(\bfX_{i,NW}) = 1,
\]
so that
\[
   \grad \phi_{i;1}^{NW} \equiv \frac{1}{h}(1,0), \qquad
   \grad \phi_{i;2}^{NW} \equiv -\frac{1}{h}(1,1), \qquad
   \grad \phi_{i;3}^{NW} \equiv \frac{1}{h}(0,1), \qquad
   \mbox{on } T_i^{NW}.
\]
\item
On the $N$-triangle $T_i^{N}$ we construct the linear basis functions $\phi_{i;1}^{N}$,
$\phi_{i;2}^{N}$, $\phi_{i;3}^{N}$ defined by
\[
   \phi_{i;1}^{N}(\bfX_i) = 1, \quad \phi_{i;1}^{N}(\bfX_{i,NW}) = 0, \quad \phi_{i;1}^{N}(\bfX_{i,N}) = 0,\\
   \phi_{i;2}^{N}(\bfX_i) = 0, \quad \phi_{i;2}^{N}(\bfX_{i,NW}) = 1, \quad \phi_{i;2}^{N}(\bfX_{i,N}) = 0,\\
   \phi_{i;3}^{N}(\bfX_i) = 0, \quad \phi_{i;3}^{N}(\bfX_{i,NW}) = 0, \quad \phi_{i;3}^{N}(\bfX_{i,N}) = 1,
\]
so that
\[
   \grad \phi_{i;1}^{N} \equiv -\frac{1}{h}(0,1), \qquad
   \grad \phi_{i;2}^{N} \equiv -\frac{1}{h}(1,0), \qquad
   \grad \phi_{i;3}^{N} \equiv \frac{1}{h}(1,1), \qquad
   \mbox{on } T_i^{N}.
\]

%%%%%%%%%%%%%%%%%%%%%%%%%%%%%%%%%
\subsection{Gradients of $p^h$}\label{subsec:App2}
Here we provide the gradient of the solution $p^h\in W^h$ of \eqref{PoissonFEM}, constructed in Section \ref{subsec:FEM-Poisson}.
Since $p^h$ is continuous on $\Omega$ and linear on each triangle in $\calT^h$,
it is represented by its vertex values $P_i^h := p^h(\bfX_i)$, $i\in \mathbb{V}$.
Then, for any interior node $i\in \mathbb{V}$ we readily have
\[
   \grad p^h = \frac{1}{h}
  \begin{cases}
    ( P^h_{i,E} - P^h_i, P^h_{i,N} - P^h_i) \qquad\mbox{on } T_i^{NE}, \\
    ( P^h_{i,E} - P^h_i, P^h_i - P^h_{i,SE}) \qquad\mbox{on } T_i^{SE}, \\
    ( P^h_{i,SE} - P^h_i, P^h_i - P^h_{i,S}) \qquad\mbox{on } T_i^{S}, \\
    ( P^h_i - P^h_{i,W}, P^h_i - P^h_{i,S}) \qquad\mbox{on } T_i^{SW}, \\
    ( P^h_i - P^h_{i,W}, P^h_{i,NW} - P^h_i) \qquad\mbox{on } T_i^{NW}, \\
    ( P^h_i - P^h_{i,NW}, P^h_{i,N} - P^h_i) \qquad\mbox{on } T_i^{N}.
   \end{cases}
\]

%%%%%%%%%%%%%%%%%%%%%%%%%%%%%%%%%
\subsection{Explicit calculation for \eqref{CforA}}\label{subsec:App3}
Finally, we provide the detailed calculation for the identity \eqref{CforA}.
Noting that $\psi^h_i$ is supported on $U_i = T_i^{NE} \cup T_i^{SE} \cup T_i^{S} \cup T_i^{SW} \cup T_i^{NW} \cup T_i^{N}$,
and taking into account the results listed in Sections \ref{subsec:App1} and \ref{subsec:App2}, we have 
\[
   \int_{T_i^{NE}} \grad p^h \cdot \left(rI + \Q_0^h[C]\right) \grad \psi^h_i \d \x &=& \frac{r+C_{i}^E}{2} \left( P_i^h - P^h_{i,E} \right) + \frac{r+C_{i}^N}{2} \left( P_i^h - P^h_{i,N} \right),\\
   \int_{T_i^{SE}} \grad p^h \cdot \left(r I+ \Q_0^h[C]\right) \grad \psi^h_i \d \x &=& \frac{r+C_{i}^E}{2} \left( P_i^h - P^h_{i,E} \right),\\
   \int_{T_i^{S}} \grad p^h \cdot \left(r I+ \Q_0^h[C]\right) \grad \psi^h_i \d \x &=& \frac{r+C_{i}^S}{2} \left( P_i^h - P^h_{i,S} \right),\\
   \int_{T_i^{SW}} \grad p^h \cdot \left(r I+ \Q_0^h[C]\right) \grad \psi^h_i \d \x &=& \frac{r+C_{i}^W}{2} \left( P_i^h - P^h_{i,W} \right) + \frac{r+C_{i}^S}{2} \left( P_i^h - P^h_{i,S} \right),\\
   \int_{T_i^{NW}} \grad p^h \cdot \left(r I+ \Q_0^h[C]\right) \grad \psi^h_i \d \x &=& \frac{r+C_{i}^W}{2} \left( P_{i}^h - P^h_{i,W} \right),\\
   \int_{T_i^{N}} \grad p^h \cdot \left(rI + \Q_0^h[C]\right) \grad \psi^h_i \d \x &=& \frac{r+C_{i}^N}{2} \left( P_i^h - P^h_{i,N} \right).
\]
Summing up, we arrive at
\[
   \int_\Omega  \grad p^h \cdot (rI+\Q_0^h[C]) \grad \psi^h_i \d \x &=&
      (r+C_{i}^E) \left( P_i^h - P^h_{i,E} \right) + (r+C_{i}^N) \left( P_i^h - P^h_{i,N} \right)
      \\ &+&  (r+C_{i}^W) \left( P_i^h - P^h_{i,W} \right) + (r+C_{i}^S) \left( P_i^h - P^h_{i,S} \right),
\]
which is \eqref{CforA}.
\end{itemize}

%%%%%%%%%%%%%%%%%%%%%%%%%%%%%%%%%
%% Ack
%%%%%%%%%%%%%%%%%%%%%%%%%%%%%%%%%
\section*{Acknowledgments}
LMK was supported by the UK Engineering and Physical Sciences Research Council (EPSRC) grant
EP/L016516/1 and the German National Academic Foundation (Studienstiftung des Deutschen Volkes). 

%%%%%%%%%%%%%%%%%%%%%%%%%%%%%%%%%%%%%%%%%%%%%%%%%%%%%%


\begin{thebibliography}{99}

\bibitem{AAFM} G. Albi, M. Artina, M. Fornasier and P. Markowich: \emph{Biological transportation networks: modeling and simulation.}
Analysis and Applications. Vol. 14, Issue 01 (2016).

\bibitem{Braides}
A. Braides: \emph{A handbook of $\Gamma$-convergence.} In:
Eds. M. Chipot, P. Quittner,
Handbook of Differential Equations: Stationary Partial Differential Equations,
North-Holland, Vol. 3, pp. 101--213 (2006).

\bibitem{bookchapter}
G. Albi, M. Burger, J. Haskovec, P. Markowich, and M. Schlottbom: \emph{Continuum Modelling of Biological Network Formation.}
In: N. Bellomo, P. Degond, and E. Tamdor (Eds.), 
{\em Active Particles Vol.I - Theory, Models, Applications}, Series: Modelling and Simulation in Science and Technology, Birkh\"auser-Springer (Boston), 2017.
%{\tt doi:} 10.1007/978-3-319-49996-3.

\bibitem{Carmo}
M.P. Do Carmo: \emph{Differential Geometry of Curves and Surfaces.} Prentice-Hall, Englewood Cliffs, NJ (1976).

\bibitem{Ciarlet}
P. G. Ciarlet: \emph{The finite element method for elliptic problems.}
North-Holland, Amsterdam (1978).

\bibitem{DalMaso}
G. Dal Maso: \emph{An Introduction to $\Gamma$-Convergence.}
Birkh\"auser (1993).

\bibitem{Gross-Yellen}
J.L. Gross and J. Yellen: \emph{Handbook of Graph Theory.} CRC Press (2004).

\bibitem{HKM} J. Haskovec, L. M. Kreusser and P. Markowich: \emph{ODE and PDE based modeling of biological transportation networks.}
Preprint, arXiv:1805.08526 (2018).

\bibitem{HMP15} J. Haskovec, P. Markowich and B. Perthame: \emph{Mathematical Analysis of a PDE System for Biological Network Formation.}
Comm. PDE 40:5, pp. 918-956 (2015).

\bibitem{HMPS16} J. Haskovec, P. Markowich, B. Perthame and M. Schlottbom: Notes on a PDE system for biological network formation.
Nonlinear Analysis 138 (2016), pp. 127--155.  %doi: 10.1016/j.na.2015.12.018

\bibitem{HMR} J. Haskovec, P. Markowich and H. Ranetbauer: \emph{A mesoscopic model of biological transportation networks.}
Preprint, arXiv:1806.00120 (2018).

\bibitem{Hu} D. Hu: \emph{Optimization, Adaptation, and Initialization of Biological Transport Networks.}
Notes from lecture (2013).

\bibitem{Hu-Cai} D. Hu and D. Cai: \emph{Adaptation and Optimization of Biological Transport Networks.}
Phys. Rev. Lett. 111 (2013), 138701.

\bibitem{Murray} C. Murray: \emph{The physiological principle of minimum work. I. The vascular system and the cost of blood volume.}
Proc. Natl. Acad. Sci. USA. 12, 207 (1926).

\end{thebibliography}
\end{document}